\documentclass[10pt]{amsart}
\usepackage{amsmath}
\pdfoutput=1
\usepackage{paralist}
\usepackage{graphicx} 
\usepackage{epstopdf}
\usepackage[colorlinks=true]{hyperref}
\usepackage{float}
\usepackage{color,comment}
\usepackage{caption}
\usepackage{wrapfig}
\usepackage{multirow}
\usepackage{tabularx}
\usepackage{booktabs}
\usepackage{subfigure}
\usepackage{enumitem}
\newtheorem{theorem}{Theorem}[section]

\newtheorem{lemma}[theorem]{Lemma}
\newtheorem{proposition}[theorem]{Proposition}

\theoremstyle{definition}

\newtheorem{remark}{Remark}


\begin{document}
\title[Mutual Interference with Generalized Holling type response]{DYNAMICS OF A PREDATOR-PREY MODEL WITH GENERALIZED HOLLING TYPE FUNCTIONAL RESPONSE AND MUTUAL INTERFERENCE}
\author[Antwi-Fordjour, Parshad, Beauregard]{}
\maketitle

\centerline{\scshape Kwadwo Antwi-Fordjour $^1$, Rana D. Parshad$^2$, Matthew A. Beauregard$^3$}

\medskip
{\footnotesize
  \centerline{ 1) Department of Mathematics and Computer Science,}
 \centerline{ Samford University,}
 \centerline{ Birmingham, AL 35229, USA}
}
\medskip

{\footnotesize
 \centerline{ 2) Department of Mathematics,}
 \centerline{ Iowa State University,}
 \centerline{ Ames, IA 50011, USA}
   }
\medskip

   {\footnotesize
 \centerline {3) Department of Mathematics and Statistics,}
 \centerline{ Stephen F. Austin State University,}
 \centerline{ Nacogdoches, TX 75962, USA}
   }

\begin{abstract}
Mutual interference and prey refuge are important drivers of predator-prey dynamics. The ``exponent" or degree of mutual interference has been under much debate in theoretical ecology.
In the present work, we investigate the interplay of the mutual interference exponent, on the behavior of a predator-prey model with a generalized Holling type functional response. We investigate stability properties of the system and derive conditions for the occurrence of saddle-node and Hopf-bifurcations. A sufficient condition for extinction of the prey species has also been derived for the model. In addition, we investigate the effect of a prey refuge on the population dynamics of the model and derive conditions for the prey refuge that would yield persistence of populations. We provide additional verification our analytical results via numerical simulations.
Our findings are in accordance with classical experimental results in ecology \cite{G34}, that show that extinction of predator and prey populations is possible in a finite time period - but that bringing in refuge can effectively  
cause persistence. 

\end{abstract}

\section{Introduction}

Predator-prey dynamics form the corner stone of ecosystems. Mathematical models for such interactions goes back to the work of Lokta, Volterra, Holling and Gause \cite{B75, G34, H59, L25}. Holling's classical work proposes that a predators feeding rate depends solely on the prey density, and is modeled essentially by a saturating function called the functional response, described via $p(x) = \frac{f(x)}{1+hf(x)}$, where $h$ is the handling time of one prey item and $f(x)$ is a function of prey density $x$. Typically $f$ is smooth, making the response $p$ smooth, and depending on the form of $f$ we have Holling type II, III, IV responses \cite{K07, H59}.
The response, $p(x) = \frac{f(x)}{1+hf(x)}$, can be considered a special form of 2 different response types.
$p(x) = \frac{(f(x))^{p}}{1+h(f(x))^{p}}$, or $p(x) = \left(\frac{f(x)}{1+hf(x)}\right)^{m_{1}}$. When $p=1$, or $m_{1}=1$, we recover the classical response posed earlier.

An interesting subclass of these general responses are the cases when $0<p<1$ or $0< m_{1} < 1$. In these cases $p(x)$ is \emph{non-smooth}, causing various difficulties in the mathematical analysis of these systems. For example, linearization about the trivial steady state is no longer possible \cite{S97}. Such responses were considered by Sugie \cite{S97, S99} and more recently by Braza \cite{B12}. However these works miss a key dynamic inherent in such models - that of finite time extinction. Note, in these cases one might ask what real ecological scenarios do these models represent. The work by Sugie, proposed the $(f(x))^{p}$ term, as indicative of a predator which is highly efficient and with a high attack rate.

Mutual interference is defined as the behavioral interactions among feeding organisms, that reduce the time that each individual spends obtaining food, or the amount of food each individual consumes \cite{Hassell71, Hassell75, D75, F79}. Some of the earliest work on mutual/predator interference, was initiated by Erbe \cite{E85}, in which the mutual interference is modeled as $\frac{f(x)}{1+hf(x)}y^{m}$, where $y$ is the predator density and $0 < m < 1$. The exact value of the exponent $m$ has been under much debate in ecology \cite{DV11}. Various authors describe the response $p(x) = \left(\frac{f(x)}{1+hf(x)}\right)^{m_{1}}$ in terms of mutual interference. This direction was first considered by Upadhyay and Rao \cite{U09}. However, an ecological motivation, to the best of our knowledge is not provided. We are motivated by certain theoretical ecology directions \cite{L12, M04, R05}, and interpret
$p(x) = \left(\frac{f(x)}{1+hf(x)}\right)^{m_{1}}$, $0 < m_{1} < 1$, as a predator with a greater feeding rate, or a more aggressive predator, than one which is modeled in the classical scenario - that is when $m_{1} = 1$. This is clear from simple comparison, $p(x)|_{m_{1}=1} < p(x)|_{0 < m_{1} < 1}$, $\forall x > 0$.

Prey refuge, and its role in predator-prey communities has also been extremely well investigated, since the seminal work of Kar \cite{K05}. Refuge is defined as any strategy taken by prey to avoid predation, such as shelter, dispersal, mimicry and camouflage \cite{F06}. It can have strong influence on predator-prey communities \cite{P16, K11} - often stabilizing systems, which are otherwise doomed for extinction. It is thus an important ingredient in ecosystem balance and diversity \cite{W13}. However, the effect of refuge on non-smooth systems such as the affore mentioned ones, remains less investigated \cite{W13}. The well known experiments of Gause find in contradiction to the predictions of classical predator-prey models, that there is a distinct chance for the predator and prey populations to die out - unless the prey is provided with refuge \cite{G34}. Non-smooth systems such as when $0<m,m_{1} <1$, in the affore mentioned models, enable the dynamic of finite time predator-prey extinction (such as seen in the experiments of Gause \cite{G34, K11}) - however, to the best of our knowledge, the effect of prey refuge on these systems has not been investigated.

For the purposes of this manuscript we consider the functional response $p(x) = \left(\frac{f(x)}{1+hf(x)}\right)^{m_{1}}$, and define the parameter regimes $m_{1} > 1$ as \emph{super-critical}, that is the regime where $p(x) \in C^{k}, \forall k$, and $f(x)$ is a polynomial function. We define $m_{1} = 1$ as \emph{critical}, recovering the classical case from the literature. Lastly we define $0< m_{1} < 1$ as \emph{sub-critical}, that is the regime where $p(x)$ looses smoothness, and is not even Lipschitz. Thus the goals of the current manuscript are:
\begin{enumerate}
\item To consider a generalized model of interference, in the \emph{sub-critical} regime; therein to investigate the phenomenon of finite time extinction, that can occur in this regime.

\item To investigate this model dynamically, including the various bifurcations that might occur;

\item To investigate the effect of prey refuge on the dynamics of this generalized model. We find that there is a critical amount of refuge that prevents finite time extinction of the prey. This is seen via theorem \ref{thm:ref1}.
\end{enumerate}
The rest of the paper is organized as follows. The mathematical formulation of the problem and mathematical preliminaries such as nonnegativity, boundedness and dissipativeness are presented in Section \ref{section:model_formulation}. Existence of equilibria, stability analysis and various local bifurcation analysis are considered in Section \ref{section:Existence_of_equilibria_Main}. In Section \ref{section:Finite_time_extinction}, we analyze the possibility of finite time extinction of the prey population. We investigate the effect of prey refuge in Section \ref{section:Effect_of_prey_refuge}. Additionally, stability analysis and various local bifurcation analysis are carried out. Numerical simulations are performed in Section \ref{section:Numerical_simulations} to correlate with some of our key analytical findings. In the last section, we present our discussions and conclusions.

\section{Model Formulation}\label{section:model_formulation}
First, we consider a general predator-prey model with mutual interference among predators of the form
\begin{equation}\label{GeneralEquation}
\left\{ \begin{array}{ll}
\dfrac{dx_{1} }{dt} &~ = x_{1} f(x_1) - w_0 g(x_1) x_{2}^{m_{{\kern 1pt} 2} } ,\\[2ex]
\dfrac{dx_{2} }{dt} &~ =-a_{2} x_{2} + w_{1} g(x_1)  x_{2}^{m_{{\kern 1pt} 2} },
\end{array}\right.
\end{equation}
where $f(x_1)$ and $g(x_1)$ are the logistic growth and the  functional response of the predator towards the prey respectively. Assume that  $0 < m_2 \leq 1$, as per literature on mutual interference \cite{Hassell71, Hassell75}.  In this paper, we consider the general logistic growth and the generalized Holling type functional response, see \cite{Upadhyay15, Upadhyay19}:
\begin{align}\label{functions}
f(x)=a_1-b_1 x,\qquad g(x)=\left(\dfrac{x}{x+d}\right)^{m_1}.
\end{align}
Assume that  $0 < m_1 \leq 1$. The assumptions placed on the functions $f$ and $g$ in \eqref{functions} are:

 \begin{enumerate}[label=(\Roman*)]
 \item $g$ is continuous for $x_1\geq 0$ and $g(0)=0$;
 \item $g$ is smooth for $x_1>0$ and $g^{\prime}(x_1)>0$ for $x_1>0$;
 \item $f$ is smooth for $x_1\geq 0$;
 \item There exists $\dfrac{a_1}{b_1}>0 $ such that $\left(x_1 - \dfrac{a_1}{b_1}\right) f(x_1)<0$ for $x_1\geq 0$, $x_1 \neq \frac{a_1}{b_1}$;
 \item For $0<m_1<1$, $g^{\prime}(0^{+}):= \lim_{x_1\to 0^+} \dfrac{g(x_1)}{x_1}=+\infty$;
 \item $g$ is not smooth for $x_1=0$ when $0<m_1<1$;
 \item The integral $\lim_{\epsilon\to 0}\displaystyle\int_{\epsilon}^{\beta}\dfrac{dx_1}{g(x_1)}$ converges for fixed  $\beta>0$.
 \end{enumerate}

Thus the predator-prey model with mutual interference and the generalized Holling type functional response becomes
\begin{equation}\label{EquationMain}
\begin{cases}
\dfrac{dx_{1} }{dt} &=a_{1} x_{1} -b_{1} x_{1}^{2} - w _{0} \left(\frac{x_{1} }{x_{1} +d} \right)^{m_{{\kern 1pt} 1} } x_{2}^{m_{{\kern 1pt} 2} } ,\\[2ex]
\dfrac{dx_{2} }{dt} &=-a_{2} x_{2} + w_{1} \left(\frac{x_{1} }{x_{1} +d} \right)^{m_{{\kern 1pt} {\kern 1pt} 1} } x_{2}^{m_{{\kern 1pt} 2} },
\end{cases}
\end{equation}
The variables and parameters used in the model are defined in Table \ref{tab:table1}.

\begin{table}[h!]
  \begin{center}
    \caption{List of parameters used in the model \eqref{EquationMain}. All parameters considered are positive constants.}
    \label{tab:table1}
    \begin{tabular}{@{}l l@{}}
      \toprule
      Variables/\\Parameters & Description \\
      \midrule
$x_{1}$                         & Prey population\\
$x_{2}$                         & Predator population\\
$t$                                & Time\\
$a_1$                           & Per capita rate of self-reproduction for the prey \\
$a_2$                           & Intrinsic death rate of the predator population \\
$w _{0}$                       & Maximum rate of per capita removal of prey \\
$w _{1}$                       & Measure efficiency of biomass conversion from prey to predator \\
 $b_{1}$                        & Death rate of prey population due to intra-species competition \\
 $d$                              & Half saturation constant\\
 $1/m_1$            &Predators feeding intensity \\
 $m_2$                          & Mutual interference exponent \\
      \bottomrule
    \end{tabular}
  \end{center}
\end{table}
~

\subsection{Mathematical Preliminaries}~\\

There are essential properties that a mathematical model must exhibit in order to obtain realistic solutions.  In particular, it is important to guarantee positivity of the populations. Likewise, boundedness of the total population is another important feature of a realistic model.  In this section, we present guarantee positivity, boundedness, and dissipativeness of the mathematical model \eqref{EquationMain}.

\subsubsection{Positivity and Boundedness}

The nonnegativity of populations generated by the mathematical model \eqref{EquationMain} is clearly important to make biological sense.  In addition, positivity implies survival of the populations over the temporal domain.  The boundedness of populations ensures that no population supercedes unrealistic values in time. In particular, boundedness guarantees that the total population does not grow beyond an exponential rate for an unbounded interval.  Guaranteeing both of these features makes strides to showing the feasibility of a mathematical model for describing population behavior.
\begin{lemma}\label{positivity}
Consider the following region $\mathbb{R}_+^2=\{(x_1,x_2):x_1\geq 0, x_2\geq 0\}$, then all solutions $(x_1(t),x_2(t))$ of model \eqref{EquationMain} with initial conditions $x_1(0)>0,~x_2(0)>0$ are nonnegative for all $t\geq 0$.
\end{lemma}
\begin{proof}
The proof of Lemma \eqref{positivity} follows from the proof of Theorem $3.1$ in \cite{Upadhyay15}.~\\
\end{proof}
\begin{lemma}\label{boundedness}
All solutions $(x_1(t),x_2(t))$ of model \eqref{EquationMain} with initial conditions $x_1(0)>0,~x_2(0)>0$ are bounded.
\end{lemma}
\begin{proof}

Let us define the function
$Q\left(x_1(t),x_2(t)\right)=x_1(t)+x_2(t).$
Then
\begin{align*}
\dfrac{dQ}{dt}& = \dfrac{dx_1}{dt}+\dfrac{dx_2}{dt}\\
& = a_{1} x_{1} -b_{1} x_{1}^{2} - w _{0} \left(\frac{x_{1} }{x_{1} +d} \right)^{m_{{\kern 1pt} 1}} x_{2}^{m_{{\kern 1pt} 2}}-a_{2} x_{2} + w_{1} \left(\frac{x_{1} }{x_{1} +d} \right)^{m_{{\kern 1pt} 1}} x_{2}^{m_{{\kern 1pt} 2}}.
\end{align*}
Let $\delta$ be a positive constant such that $\delta\leq a_2$ and suppose $w_0\geq w_1$, then we obtain,
\begin{align*}
\dfrac{dQ}{dt}+\delta Q & =  a_{1} x_{1} -b_{1} x_{1}^{2} - w _{0} \left(\frac{x_{1} }{x_{1} +d} \right)^{m_{{\kern 1pt} 1} } x_{2}^{m_{{\kern 1pt} 2} }-a_{2} x_{2} + w_{1} \left(\frac{x_{1} }{x_{1} +d} \right)^{m_{{\kern 1pt} {\kern 1pt} 1} } x_{2}^{m_{{\kern 1pt} {\kern 1pt} 2} }\\&
 +\delta (x_1+x_2)\\
& = a_1x_1-b_{1} x_{1}^{2}- (w _{0}-w_{1}) \left(\frac{x_{1} }{x_{1} +d} \right)^{m_{{\kern 1pt} 1}} x_{2}^{m_{{\kern 1pt} 2}}-(a_2-\delta)x_2+\delta x_1\\
& \leq (a_1+\delta)x_1-b_{1} x_{1}^{2}~\leq~ \dfrac{(a_1+\delta)^2}{4b_1}.
\end{align*}
Taking $W_1=\dfrac{(a_1+\delta)^2}{4b_1}$ and applying the theory on differential inequality, we obtain
\[0\leq Q\left(x_1(t),x_2(t)\right)\leq \dfrac{W_1(1-e^{-\delta t})}{\delta}+ Q\left(x_1(0),x_2(0)\right)e^{-\delta t}, \]
which implies
\begin{equation}\label{Limsupbound}
\limsup_{t\rightarrow \infty} Q\left(x_1(t),x_2(t)\right)\leq \dfrac{W_1}{\delta}.
\end{equation}
By \eqref{Limsupbound} and Lemma \eqref{positivity}, all solutions of \eqref{EquationMain} with initial conditions $x_1(0)>0,~x_2(0)>0$ will be contained in the region
\[\Theta=\{(x_1,x_2)\in\mathbb{R}_+^2:Q\left(x_1(t),x_2(t)\right)\leq \dfrac{W_1}{\delta}+\epsilon, \;\text{for any}\;\epsilon>0\}. \]
The proof is complete.
\end{proof}~\\

\subsubsection{\bf{Dissipativeness}}~\\

In the previous section, it was shown that the total population remains positive and bounded for all time.  Here, we showed that the individual populations are all bounded from above. In such a situation, we say that the model is dissipative.
\begin{lemma}\label{dissipative}
The system \eqref{EquationMain} is dissipative.
\end{lemma}
\begin{proof}
From the first system of \eqref{EquationMain}
\begin{align*}
\dfrac{dx_{1} }{dt} &=a_{1} x_{1} -b_{1} x_{1}^{2} - w _{0} \left(\frac{x_{1} }{x_{1} +d} \right)^{m_{{\kern 1pt} 1}} x_{2}^{m_{{\kern 1pt} 2}} ,\\
 &\leq a_{1} x_{1} -b_{1} x_{1}^{2}.
\end{align*}
This implies that
\begin{align}\label{limsup x1}
\limsup_{t\rightarrow \infty} x_1(t)\leq\dfrac{a_1}{b_1}.
\end{align}
The inequality \eqref{limsup x1} gives that for arbitrary small $\epsilon_1>0,$ there exist a real number $T>0$ such that
\begin{align}\label{limsup x1 1}
x_1(t)\leq\dfrac{a_1}{b_1}+\epsilon_1,\;\text{for all}\; t\geq T_1.
\end{align}
Using \eqref{limsup x1 1}, we obtain  for all $t\geq T_1$,
\begin{align*}
\dfrac{d}{dt}\left(x_1+\dfrac{w_0 x_2}{w_1} \right)&=\dfrac{dx_1}{dt}+\dfrac{w_0}{w_1}\dfrac{dx_2}{dt}\\
&=a_{1} x_{1}  -b_{1} x_{1}^{2} -\dfrac{w_0 a_{2}}{w_1} x_{2} \\
&\leq a_1x_1 -\dfrac{w_0 a_{2}}{w_1} x_{2}\\
&=(a_1+a_2)x_1-a_2\left(x_1+\dfrac{w_0 x_2}{w_1}  \right)\\
&\leq K_1-a_2\left(x_1+\dfrac{w_0 x_2}{w_1}  \right),
\end{align*}
where $K_1=(a_1+a_2)\left(\dfrac{a_1}{b_1}+\epsilon_1\right).$ Therefore, we obtain
\begin{align}\label{limsup x1 2}
\limsup_{t\rightarrow\infty}\left(x_1+\dfrac{w_0 x_2}{w_1} \right)\leq\dfrac{K_2}{a_2}.
\end{align}
By \eqref{limsup x1} and \eqref{limsup x1 2}, there exists a real number $K_2$ such that
$$\limsup_{t\rightarrow\infty}x_2\leq K_2. $$
Thus, for arbitrary small $\epsilon_2>0$, there exists $T_2>T_1>0$, such that for all $t\geq T_2$
$$x_2\leq K_2+\epsilon_2.$$
Therefore the model \eqref{EquationMain} is dissipative.
\end{proof}~\\
%


\section{Existence of Equilibria}\label{section:Existence_of_equilibria_Main}

In this section, we determine and analyze equilibria for our mathematical model.  Consider the solutions to the steady state equations:
\begin{align}
a_{1} x_{1} -b_{1} x_{1}^{2} - w _{0} \left(\frac{x_{1} }{x_{1} +d} \right)^{m_{{\kern 1pt} 1}} x_{2}^{m_{{\kern 1pt} 2}} &= 0  \label{equilibrium1} \\
-a_{2} x_{2} + w_{1} \left(\frac{x_{1} }{x_{1} +d} \right)^{m_{{\kern 1pt} 1}} x_{2}^{m_{{\kern 1pt} 2}} &= 0   \label{equilibrium2}
\end{align}
The above equations, \eqref{equilibrium1} and \eqref{equilibrium2},  have three types of non-negative equilibria:

 \begin{enumerate}[label=(\roman*)]
 \item The trivial equilibrium $E_0 (0,0)$;
 \item The predator-free equilibrium $E_1(a_1/b_1,0)$;
 \item The interior equilibrium $E_2(x_1^*, x_2^*)$ where $x_1^*$ and $x_2^*$ are related by
 $$x_2^*=\dfrac{w_1}{w_0 a_2}\left[a_1 x_1^* - b_1 x_1^{*2}  \right] .$$
 We have that $a_1 - b_1 x_1^*\geq 0$ since $x_1^* \geq 0$ and $x_2^* \geq 0$. The possible existence of a unique or multiple interior equilibria are shown in Fig.~\eqref{fig:Case_1}.
 \end{enumerate}

\begin{figure}[!htb]
\begin{center}
    \includegraphics[scale=.14]{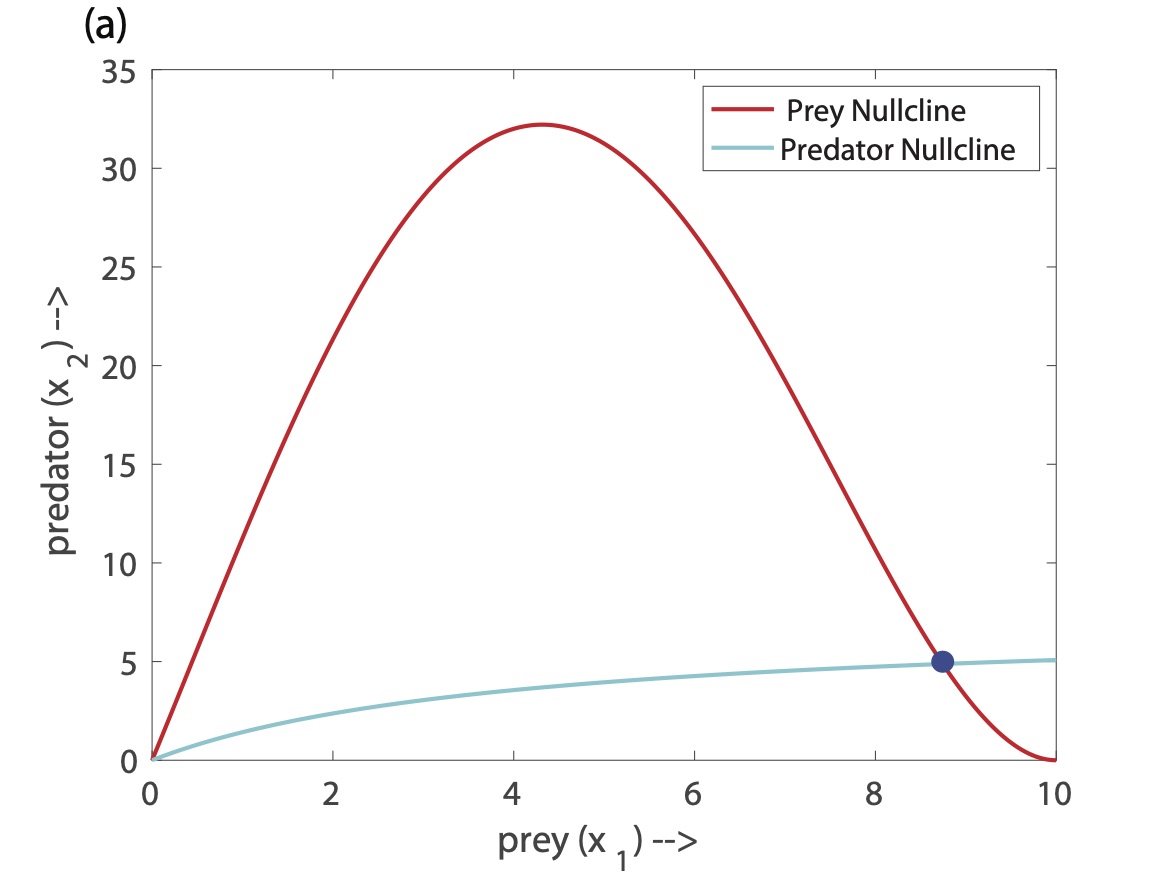}
    \includegraphics[scale=.14]{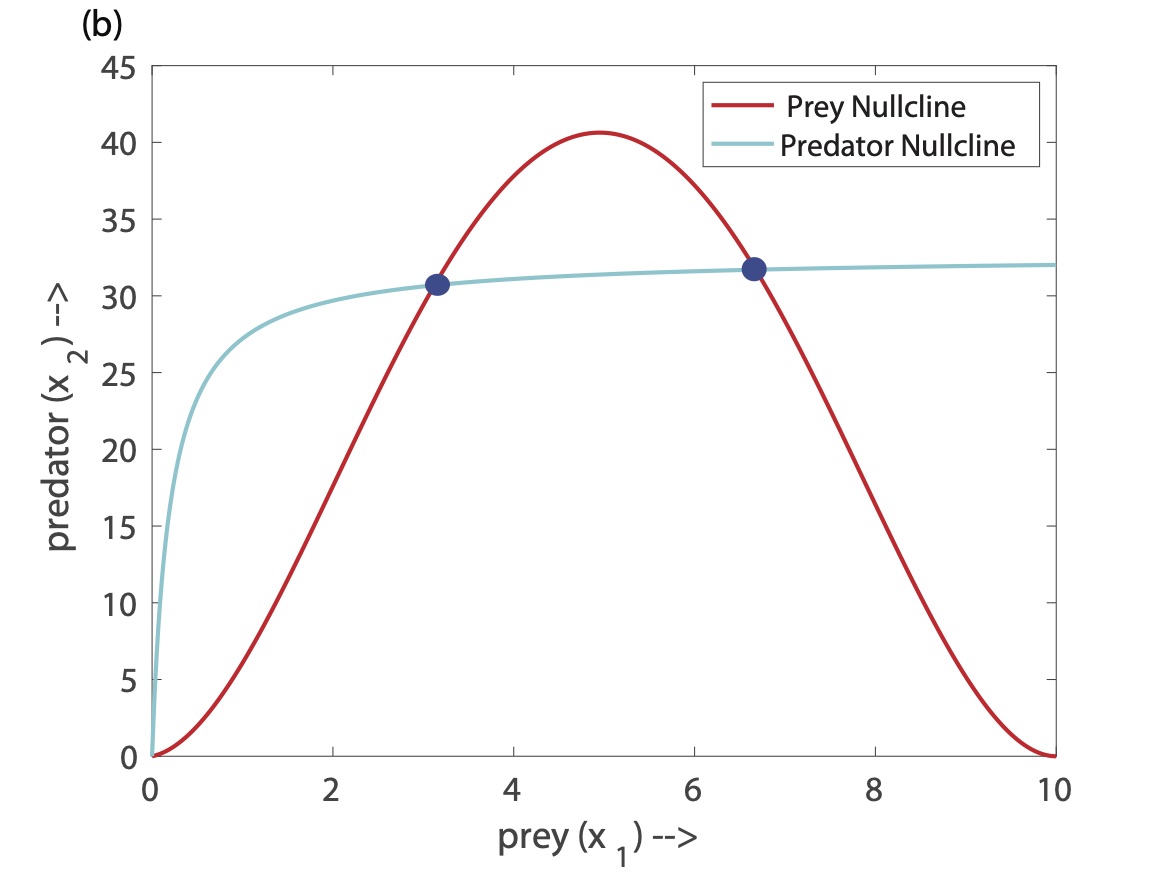}
    \includegraphics[scale=.14]{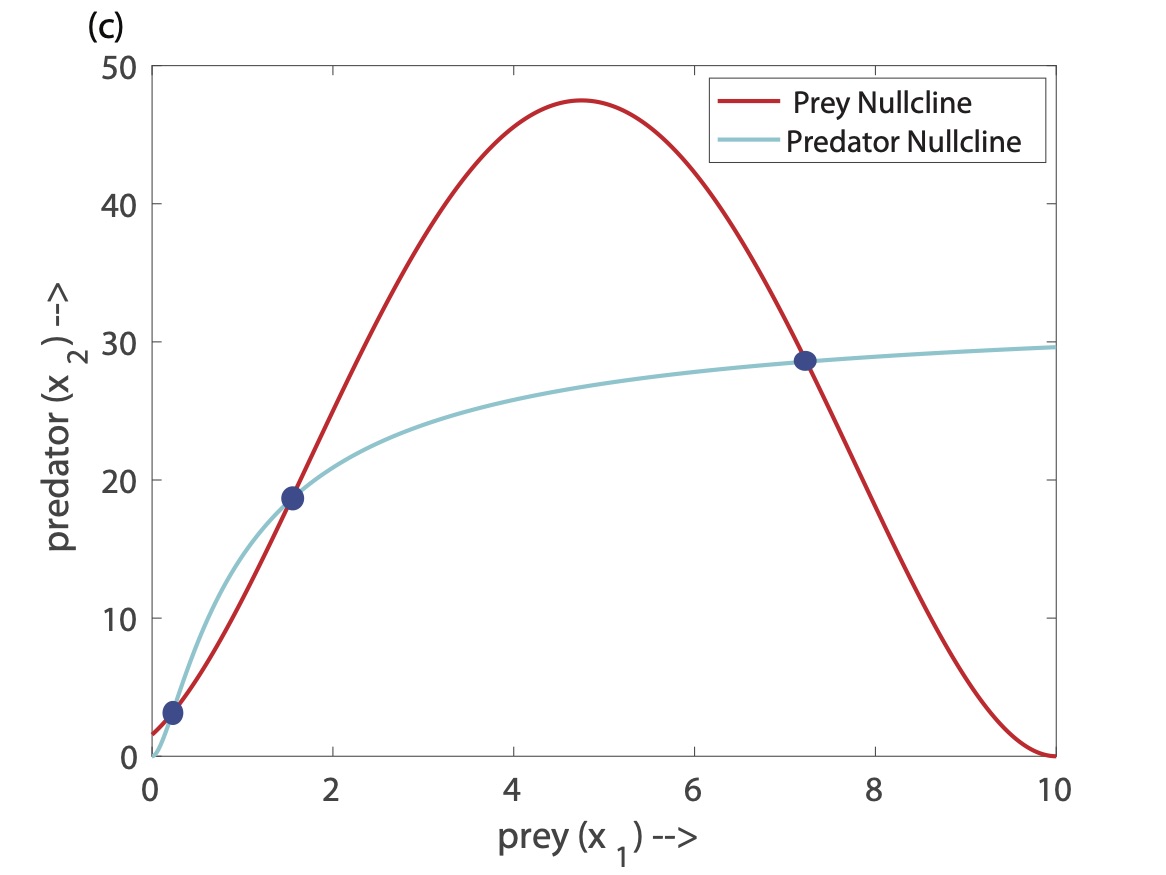}
\end{center}
 \caption{Figure (a) and (b) represent graphical illustration of the predator and prey non-trivial nullclines when $m_1=m_2=0.5$. Figure (c) represents graphical illustration of the predator and prey non-trivial nullclines when $m_1=1,~m_2=0.5$.}
      \label{fig:Case_1}
\end{figure}

\subsection{Stability Analysis of the Interior Equilibrium}~\\

The variational matrix ${\bf J^*}$ of the model \eqref{EquationMain} around the interior equilibrium
$E_2(x_1^*, x_2^*)$ is
\begin{align*}
 \bf{J^*}=  \begin{bmatrix}
      a_{11}&   a_{12}\\
      a_{21}  &  a_{22}
     \end{bmatrix},
\end{align*}
where
\begin{eqnarray*}
a_{11}&=&  a_1 - 2b_1x_1^*- {dm_1w_0{x^*}^{m_2}_2{x^*}^{m_1-1}_1 \over {(x_1^*+d)^{m_1+1}}},\\
a_{12}&=& - {m_2 w_0 {x_2^*}^{m_2-1}{x^*}^{m_1}_1 \over (x_1^*+d)^{m_1}}<0,\\
a_{21}&=&  m_1dw_1{x_2^*}^{m_2}\bigg({{x_1^*}^{m_1-1}\over (x_1^*+d)^{m_1+1}}\bigg) >0,\\
a_{22}&=& -a_2+ m_2 w_1{ x_2^*}^{m_2-1}\bigg( {x_1^*\over {x_1^*+d}}\bigg)^{m_1}.
\end{eqnarray*}
The characteristic equation corresponding to ${\bf J^*}$ evaluated at  $E_2(x_1^*, x_2^*)$ is given by
\begin{align}
\lambda^2 - \operatorname{tr} \,({\bf{J^*}}) \lambda + \det \,({\bf{J^*}}) =0,
\end{align}
where
\begin{eqnarray*}
\operatorname{tr} \,({\bf{J^*}})&=& a_{11}+a_{22}\\
&=& a_1 - a_2 - 2b_1x_1^* + m_2w_1{x^*}_2^{m_2 - 1}\bigg({x_1^*\over  x_1^*+d}\bigg)^{m_1} - dm_1w_0{x^*}_2^{m_2}\bigg({{x_1^*}^{m_1-1}\over ( x_1^*+d)^{m_1+1}}\bigg)
\end{eqnarray*}
and
\begin{eqnarray*}
\det \,({\bf{J^*}})&=& a_{11}a_{22}-a_{12}a_{21} \\
&=& \left(a_1 - 2b_1x_1^*- {dm_1w_0{x^*}^{m_2}_2{x^*}^{m_1}_1 \over {x_1^*(x_1^*+d)^{m_1+1}}} \right)  \left( - a_2 + m_2 w_1{ x_2^*}^{m_2-1}\bigg( {x_1^*\over {x_1^*+d}}\bigg)^{m_1}\right) \\
&& - \left( - {m_2 w_0 {x_2^*}^{m_2-1}{x^*}^{m_1}_1 \over (x_1^*+d)^{m_1}} \right)    \left( m_1dw_1{x_2^*}^{m_2}\bigg({{x_1^*}^{m_1-1}\over (x_1^*+d)^{m_1+1}}\bigg)   \right).
\end{eqnarray*}
Here, $\operatorname{tr} \,({\bf{J^*}})$ and $\det \,({\bf{J^*}}) $ represents the trace and determinant of the variational matrix.  Hence the stability of $E_2(x_1^*, x_2^*)$ is determined by the sign of $\det \,({\bf{J^*}}) $ and $\operatorname{tr} \,({\bf{J^*}})$.

The above results are encapsulated in the following theorem.

\begin{theorem}\label{localstabint}
The interior equilibrium $E_2(x_1^*, x_2^*)$ is locally asymptotically stable if  $\operatorname{tr} \,({\bf{J^*}}) < 0$ and $\det \,({\bf{J^*}})  > 0$ by Routh-Hurwitz stability criteria.
\end{theorem}

\begin{proof}
The proof follows directly from the above discussion and hence omitted for brevity.
\end{proof}~\\

\subsection{Global Asymptotic Stability}~\\

Considering the case where $m_2=1$ and assumptions (V)-(VII) hold for our response functions $f$ and $g$, then $E_0$ is not a saddle point and the argument using the Poincare-Bendixson theorem cannot be applied. We will describe the global behavior of the system \eqref{EquationMain} by considering the relative position of the stable and unstable separatrix of the saddle point $E_1$. We denote the stable separatrix by $W^s(E_0)$ and the unstable separatrix by $W^u(E_1)$. \\
The predator nullcline is the vertical line $x_1=x_1^*$ determined by the equation $-a_2+ w_1 g(x_1)=0$. We assume
\begin{align}\label{conditionx_1}
w_1>a_2, \qquad \dfrac{a_1}{b_1}>x_1^* := \dfrac{d~ a_2^{\frac{1}{m_1}}}{w_1^{\frac{1}{m_1}}-a_2^{\frac{1}{m_1}}}.
\end{align}

The prey nullcline is the graph of the function $y=\psi (x_1)$
\begin{align}
\psi (x_1)=\dfrac{x_1f(x_1)}{w_0g(x_1)}
\end{align}
where $f(x_1)$ and $g(x_1)$ are defined in \eqref{functions}. Clearly $\psi (\frac{a_1}{b_1})=0$ and $\psi (x_1)>0$ for $0<x_1<\frac{a_1}{b_1}$. The unique interior equilibrium $E_2(x_1^*,x_2^*)$ is the intersection of the predator and prey nullclines and it can be stable or unstable depending on the sign of $\psi^{\prime}(x_1^*)$. For $\psi^{\prime}(x_1^*)>0$, $E_2$ is a repeller (unstable) and for $\psi^{\prime}(x_1^*)<0$, $E_2$ is an attractor (or locally asymptotically stable).

\begin{remark}
The predator-free equilibrium $E_1$ turns into a stable node with the loss of the unique interior equilibrium $E_2$.
\end{remark}

Based on the non-uniqueness of the solution of model \eqref{EquationMain} when $m_1<1$ and $m_2=1$, we  have the following results.

\begin{lemma}[ Proposition 3.1 in \cite{BS19}]
Assume that $W^s(E_0)$ is above $W^u(E_1)$. If $E_2$ is a repeller, then it is surrounded by at least one limit cycle. If the system can have at most one cycle, then $E_2$ is surrounded by at least a unique limit cycle which is orbitally asymptotically stable. This limit cycle is not globally orbitally asymptotically stable, even if it is unique. If $E_2$ is an attractor and if the system has no cycles, then all orbits under $W^s(E_0)$ converge towards $E_2$. $E_2$ is not globally asymptotically stable, even if it is not surrounded by any unstable limit cycle.
\end{lemma}

\begin{proposition}
Assume $\psi^{\prime} (x_1^*)\leq 0$ and \eqref{conditionx_1} hold, then $W^s(E_0)$ is above $W^u(E_1)$ and the basin of attraction of $E_2$ is the positive region of the plane located under $W^s(E_0)$. Hence $E_2$ is not globally asymptotically stable.
\end{proposition}

\begin{proof}
The proof is similar to the proof of Proposition 4.2 in \cite{BS19} and omitted for brevity.
\end{proof}

\subsection{Saddle-Node Bifurcation Analysis}

We investigate the possibility of saddle-node bifurcation of the positive interior equilibrium $E_2$ by using the intrinsic death rate of the predator population as a bifurcation parameter.

The following theorem states the restrictions for occurrence of a saddle-node bifurcation for model \eqref{EquationMain}.

\begin{theorem}\label{saddle_node_a2}
The model \eqref{EquationMain} undergoes a saddle-node bifurcation around $E_2$ at $a_2^*$ when the system parameters satisfy the restriction $\det \,({\bf{J^*}}) = 0 $ along with the condition  $\operatorname{tr} \,({\bf{J^*}})<0$.
\end{theorem}

\begin{proof}
To validate the restriction for the occurrence of saddle-node bifurcation, we apply Sotomayor's theorem \cite{Perko13} at $a_2=a_2^*$. At $a_2=a_2^*$, it can be seen that $\det \,({\bf{J^*}}) = 0 $ and $\operatorname{tr} \,({\bf{J^*}})<0$ which indicates that the Jacobian $\,({\bf{J^*}})$ admits a zero eigenvalue. Let $U$ and $V$ be the eigenvectors corresponding to the zero eigenvalue of the matrix $\,({\bf{J^*}})$ and $\,({\bf{J^*}})^T$ respectively. We obtain that $U=(u_1,u_2)^T$ and $V=(v_1,v_2)^T$, where $u_1=-\frac{a^*_{12}u_2}{a^*_{11}}$, $v_1=-\frac{a^*_{21}v_2}{a^*_{11}}$ and $u_2,v_2 \in \mathbb{R} \setminus \{0\}$.

Furthermore, let $F=(F_1,F_2)^T$ and $X=(x_1^*,x_2^*)^T$, where $F_1, F_2$ are given by
\begin{align*}
F_1 &= a_{1} x_{1} -b_{1} x_{1}^{2} - w _{0} \left(\frac{x_{1} }{x_{1} +d} \right)^{m_{{\kern 1pt} 1}} x_{2}^{m_{{\kern 1pt} 2}}    \\
F_2 &= -a_{2} x_{2} + w_{1} \left(\frac{x_{1} }{x_{1} +d} \right)^{m_{{\kern 1pt} {\kern 1pt} 1}} x_{2}^{m_{{\kern 1pt} 2} }.
\end{align*}
Now
\begin{align*}
V^T F_{a_2}(X,a_2^*) = (v_1,v_2)(0,-x_2)^T=-v_2 x_2 \neq 0,
\end{align*}
and
\begin{align*}
V^T \left[D^2 F(X,a_2^*)(U,U)\right] \neq 0.
\end{align*}
Hence, from Sotomayor's theorem the model undergoes a saddle-node bifurcation around $E_2$ at $a_2=a_2^*$.
\end{proof}

\begin{theorem}\label{saddle_node_w1}
The model \eqref{EquationMain} undergoes a saddle-node bifurcation around $E_2$ at $w_1^*$ when the system parameters satisfy the restriction $\det \,({\bf{J^*}}) = 0 $ along with the condition  $\operatorname{tr} \,({\bf{J^*}})<0$.
\end{theorem}

\begin{theorem}\label{saddle_node_w0}
The model \eqref{EquationMain} undergoes a saddle-node bifurcation around $E_2$ at $w_0^*$ when the system parameters satisfy the restriction $\det \,({\bf{J^*}}) = 0 $ along with the condition  $\operatorname{tr} \,({\bf{J^*}})<0$.
\end{theorem}

\begin{theorem}\label{saddle_node_b1}
The model \eqref{EquationMain} undergoes a saddle-node bifurcation around $E_2$ at $b_1^*$ when the system parameters satisfy the restriction $\det \,({\bf{J^*}}) = 0 $ along with the condition  $\operatorname{tr} \,({\bf{J^*}})<0$.
\end{theorem}

\begin{theorem}\label{saddle_node_a1}
The model \eqref{EquationMain} undergoes a saddle-node bifurcation around $E_2$ at $a_1^*$ when the system parameters satisfy the restriction $\det \,({\bf{J^*}}) = 0 $ along with the condition  $\operatorname{tr} \,({\bf{J^*}})<0$.
\end{theorem}

\begin{proof}
The proof of Theorem \ref{saddle_node_w1}, Theorem \ref{saddle_node_w0}, Theorem \ref{saddle_node_b1} and Theorem \ref{saddle_node_a1}   are similar to proof in Theorem \ref{saddle_node_a2} and omitted for brevity.
\end{proof}

\subsection{Hopf-Bifurcation Analysis}

We investigate the possibility of Hopf-bifurcation of the positive interior equilibrium $E_2$ by using the per capita rate of self-reproduction for the prey, $a_1$ as a bifurcation parameter. Then, the characteristic equation corresponding to model \eqref{EquationMain} at $E_2$ is given by
\begin{align}\label{character_bifur_main}
\lambda^2 + A(a_1)\lambda + B(a_1) = 0,
\end{align}
where
$A = - \operatorname{tr} \,({\bf{J^*}}) = - (a_{11} + a_{22})$ and
$B =  \det \,({\bf{J^*}}) =  a_{11}  a_{22}- a_{12}a_{21}.$

The instability of model \eqref{EquationMain} is demonstrated via the following theorem by considering $a_1$ as a bifurcation parameter.

\begin{theorem}[Hopf-Bifurcation Theorem \cite{Murray93}]\label{Murray_birf_def}
If $A(a_1)$ and $B(a_1)$ are the smooth functions of $a_1$ in an open interval about $a_1^*\in \mathbb{R}$ such that the characteristic equation \eqref{character_bifur_main} has a pair of imaginary eigenvalues $\lambda=\zeta (a_1)\pm i \gamma (a_1)$ with $\zeta$ and $\gamma$ $\in\mathbb{R}$ so that they become purely imaginary at $a_1=a_1^*$ and $\frac{d\zeta}{d a_1}|_{a_1=a_1^*}\neq 0$, then a Hopf-bifurcation occurs around $E_2(x_1^*,x_2^*)$ at $a_1=a_1^*$ (i.e. a stability changes of $E_2(x_1^*,x_2^*)$ accompanied by the creation of a limit cycle at $a_1=a_1^*$).
\end{theorem}

\begin{theorem}\label{bifurcationMain}
The model \eqref{EquationMain} undergoes a Hopf-bifurcation around $E_2(x_1^*,x_2^*)$ when $a_1$ crosses some critical value of parameter  $a_1^*$, where
$$a_1^* = a_2 + 2b_1x_1^* - m_2w_1{x^*}_2^{m_2 - 1}\bigg({x_1^*\over  x_1^*+d}\bigg)^{m_1} + dm_1w_0{x^*}_2^{m_2}\bigg({{x_1^*}^{m_1-1}\over ( x_1^*+d)^{m_1+1}}\bigg) $$
provided:
 \begin{enumerate}[label=(\roman*)]
\item $A(a_1)=0 $,
\item $B(a_1)>0$,
\item $\dfrac{d}{da_1}\left. Re \lambda_i (a_1) \right|_{a_1=a_1^*}\neq 0$ at $a_1=a^*_1,~ i=1,2$.
\end{enumerate}
\end{theorem}

\begin{proof}
Clearly $A(a_1)$ and $B(a_1)$ are the smooth functions of $a_1$. The roots of the equation \eqref{character_bifur_main} are of the form $\lambda_1=\zeta (a_1) + i \gamma (a_1)$ and $\lambda_2=\zeta (a_1) - i \gamma (a_1)$ where $\zeta (a_1)$ and $\gamma (a_1)$ are real functions.

At $a_1=a_1^*$, the characteristic equation \eqref{character_bifur_main} reduces to
 \begin{align}\label{charac_reduced_bifur_main}
 \lambda^2 + B(a_1)=0
 \end{align}
 By solving for the roots of equation \eqref{charac_reduced_bifur_main}, we obtain $\lambda_1=i\sqrt{B}$ and
 $\lambda_2=-i\sqrt{B}$. Hence a pair of purely imaginary eigenvalues. Furthermore, we validate  the transversality condition:
 $$\dfrac{d}{da_1}Re \lambda_i (a_1) |_{a_1=a_1^*}\neq 0, i=1,2.$$
Substituting $\lambda (a_1) = \zeta (a_1) + i \gamma (a_1)$ into equation \eqref{character_bifur_main}, we obtain
 \begin{align}
 (\zeta (a_1) + i \gamma (a_1))^2 + A(a_1) (\zeta (a_1) + i \gamma (a_1)) + B(a_1)=0.
 \end{align}
Now, taking the derivative with respect to $a_1$, we get
 \begin{align*}
 2(\zeta (a_1) + i \gamma (a_1)) (\dot{\zeta} (a_1) + i \dot{\gamma} (a_1)) + A(a_1) (\dot{\zeta} (a_1) + i \dot{\gamma} (a_1))  \\
 + \dot{A} (a_1) (\zeta (a_1) + i \gamma (a_1)) + \dot{B} (a_1) = 0.
 \end{align*}
Separating the real and imaginary parts, we have
 \begin{align*}
 \dot{\zeta}(a_1) (2 \zeta (a_1) + A (a_1)) + \dot{\gamma (a_1)} (-2 \gamma (a_1)) + \dot{A} (a_1) \zeta (a_1) + \dot{B} (a_1) = 0,
 \end{align*}
which implies
  \begin{align}\label{hopf_cond1}
 \dot{\zeta} (a_1) Z_1 (a_1) - \dot{\gamma} (a_1) Z_2 (a_1) + Z_3 (a_1) = 0,
 \end{align}
and
  \begin{align*}
 \dot{\zeta}(a_1) (2 \gamma (a_1)) + \dot{\gamma} (a_1) (2 \zeta (a_1) + A (a_1)) + \dot{A} (a_1) \gamma (a_1) =0,
 \end{align*}
which implies
  \begin{align}\label{hopf_cond2}
 \dot{\zeta} (a_1) Z_2 (a_1) + \dot{\gamma} (a_1) Z_1 (a_1) + Z_4 (a_1) = 0,
 \end{align}
where $Z_1 (a_1)= 2 \zeta (a_1) + A (a_1)$, $Z_2 (a_1)= 2 \gamma (a_1)$, $Z_3 (a_1)=  \dot{A} (a_1) \zeta (a_1) + \dot{B} (a_1) $ and $Z_4 (a_1)=\dot{A} (a_1) \gamma (a_1) $.

Multiplying equation \eqref{hopf_cond1} by $Z_1 (a_1)$ and equation \eqref{hopf_cond2} by $Z_2 (a_1)$ and then adding them, we obtain
 \begin{align}\label{hopf_cond3}
 (Z_1^2 (a_1) + Z_2^2 (a_1)) \dot{\zeta} (a_1) + Z_1 (a_1) Z_3 (a_1) + Z_2 (a_1) Z_4 (a_1) =0,
 \end{align}
thus solving for $\dot{\zeta} (a_1)$ from equation \eqref{hopf_cond3} and at $a_1=a_1^*$,
\begin{align*}
\dfrac{d}{da_1}Re \lambda_i (a_1) |_{a_1=a_1^*} = \dot{\zeta} (a_1^*)=\dfrac{-\left[ Z_1 (a_1^*) Z_3 (a_1^*) + Z_2 (a_1^*) Z_4 (a_1^*) \right]}{Z_1^2 (a_1^*) + Z_2^2 (a_1^*)}.
 \end{align*}
It is easy to verify that $Z_1 (a_1^*) Z_3 (a_1^*) + Z_2 (a_1^*) Z_4 (a_1^*) \neq 0$ and $Z_1^2 (a_1^*) + Z_2^2 (a_1^*) \neq 0$ which implies  $\frac{d}{da_1}Re \lambda_i (a_1) |_{a_1=a_1^*} \neq 0$. Hence, a Hopf-bifurcation occurs around $E_2(x_1^*,x_2^*)$ at $a_1=a_1^*$.

\end{proof}

\section{Finite Time Extinction}\label{section:Finite_time_extinction}

An interesting property of \eqref{EquationMain} is that the population of prey may go extinct in finite time for carefully chosen initial conditions. Therefore, while solutions do remain nonnegative, it is clear that the populations may become extinct and no longer persist. 

\begin{theorem}\label{FiniteTimeTheorem}
Consider the predator-prey system given by \eqref{EquationMain}. The solution $x_1(t)$  to the prey equation $x_1(t)$ with initial conditions $x_1(0)>0,~x_2(0)>0$ can go extinct in finite time, for sufficiently chosen initial conditions.
\end{theorem}

\begin{proof}

Consider the substitution $x_{1} = 1/u$ in the prey equation of \eqref{EquationMain}. This yields the following system:
\begin{equation}\label{Eu}
\left\{ \begin{array}{rl}
\dfrac{dx_{1} }{dt} = \dfrac{-1}{u^{2}} \dfrac{du }{dt}&=a_{1} \dfrac{1}{u} -b_{1}( \dfrac{1}{u})^{2} - w _{0} \left(\frac{\frac{1}{u} }{\frac{1}{u} +d} \right)^{m_{{\kern 1pt} 1}} x_{2}^{m_{{\kern 1pt} 2}} ,\\[2ex]
\dfrac{dx_{2} }{dt} &=-a_{2} x_{2} + w_{1} \left(\frac{\frac{1}{u} }{\frac{1}{u} +d} \right)^{m_{{\kern 1pt} 1}} x_{2}^{m_{{\kern 1pt} 2} }.
\end{array}\right.
\end{equation}
This system can be simplified into the system in $u, x_{2}$:
\begin{equation}\label{Eu1}
\begin{cases}
 \dfrac{du }{dt} &=-a_{1}u + b_{1} + w _{0} \dfrac{u^{2}}{(1+du)^{m_{1}}}  x_{2}^{m_{2} } ,\\[2ex]
\dfrac{dx_{2} }{dt} &=-a_{2} x_{2} + w_{1}  \dfrac{1}{(1+du)^{m_{1}}}  x_{2}^{m_{2} },
\end{cases}
\end{equation}
Note, via positivity
\begin{equation}
\frac{dx_{2} }{dt} \geq -a_{2} x_{2}.
\end{equation}
Thus,
\begin{equation}
x_{2}  \geq x_{2}(0)e^{-a_{2}t}.
\end{equation}
Also, via positivity we have the inequality,
\begin{equation}
 \frac{du }{dt} \geq -a_{1}u  + w _{0} \frac{u^{2}}{(1+du)^{m_{1}}}  x_{2}^{m_{2} } \geq -a_{1}u  + w _{0} \frac{u^{2}}{(1+du)^{m_{1}}}  (x_{2}(0)e^{-a_{2}t})^{m_{2} }.
\end{equation}
Note that the solution to the differential equation
\begin{equation}
 \frac{d \tilde{u} }{dt} \geq -a_{1}\tilde{u}  + w _{0} \frac{\tilde{u}^{2}}{(1+d\tilde{u})^{m_{1}}}
\end{equation}
will blow up in a finite time, $T^{*}(u_{0})<\infty$, as long as the initial data $\tilde{u}(0)=u_{0}$ satisfies,
\begin{equation}
\label{d1}
 a_{1}u_{0}(1+du_{0})^{m_{1}}  \leq w_{0} u_{0}^{2}.
\end{equation}
Now if we choose $x_{2}(0) \gg 1$, such that
\begin{equation}
(x_{2}(0)e^{-a_{2}t})^{m_{2} } >1, \ t \in [0, T^{*}],
\end{equation}
then $u \geq \tilde{u}$ on $[0, T^{*}]$, and must blow-up in finite time, at some $T^{**} < T^{*}$, by comparison, if $u_{0}$ is chosen to satisfy \eqref{d1} . Therefore,
\[ \lim_{t \rightarrow T^{**} < \infty} u \rightarrow \infty \]
which implies
\[ \lim_{t \rightarrow T^{**} < \infty} x_{1} = \lim_{t \rightarrow T^{**} < \infty} \frac{1}{u} = \frac{1}{ \lim_{t \rightarrow T^{**} < \infty} u} \rightarrow 0, \]
but that implies $x_{1}(t)$ goes extinct in finite time for $x_{2}(0)$ chosen large enough and $$(x_{1}(0))^{1-m_{1}}(x_{1}(0)+d)^{m_{1}} \leq \dfrac{w_{0}}{a_{0}}.$$

\end{proof}

\section{The effect of prey refuge}\label{section:Effect_of_prey_refuge}

In the previous section it was shown that the prey population may go extinct in finite time.  Therefore, we seek to investigate the effect of protecting the prey from predation with their habitat.  The aim is to provided avenues for which the prey population will persist.  Here, using a similar ideas from \cite{P16}, we introduce a prey refuge. A discussion of how a habitat controller may create a prey refuge is provided in section 7 of \cite{P16}.

Essentially, one must protect a constant proportion of prey by replacing the predation term $g(x_{1})$ by $g(r x_{1})$, where $0\leq r \leq 1$. Here, $r$ is a refuge parameter, such that if $r=0$ then complete protection of the prey is provided while $r=1$ implies no protection and the original system \eqref{GeneralEquation} is recovered. Thus, we write the following system that models prey refuge as
\begin{equation}\label{refuge}
\begin{cases}
\dfrac{dx_{1} }{dt} &=a_{1} x_{1} -b_{1} x_{1}^{2} - w _{0} \left(\dfrac{rx_{1} }{rx_{1} +d} \right)^{m_{{\kern 1pt} 1} } x_{2}^{m_{{\kern 1pt} 2} } ,\\[2ex]
\dfrac{dx_{2} }{dt} &=-a_{2} x_{2} + w_{1} \left(\dfrac{rx_{1} }{rx_{1} +d} \right)^{m_{{\kern 1pt} 1} } x_{2}^{m_{{\kern 1pt} 2} },
\end{cases}
\end{equation}
We now state our first result concerning prey refuge,

\begin{theorem}
\label{thm:ref1}
Consider the predator-prey system given by \eqref{refuge} for $r=1$, that is no refuge. For any initial conditions $\frac{a_{1}}{b_{1}} > x_1(0)>0,~x_2(0)>0$ s.t. the solution $x_1(t)$ to the prey equation goes extinct in finite time, there exists a refuge $r^{*}(x_{1}(0), a_{1}, b_{1}, d, w_{0}, m_{1}, m_{2}, K_{2}) > 0$ s.t. for any $r < r^{*}$ the solution $x_1(t)$ to the prey equation does not go extinct in finite time.
\end{theorem}
\begin{proof}

The proof follows in similar fashion to the proof of Theorem \eqref{FiniteTimeTheorem}.  Hence, we make the substitution $x_{1} = 1/u$ in the prey equation in \eqref{refuge}.  This results in:
\begin{equation}\label{Eu113}
\begin{cases}
 \dfrac{du }{dt} &=-a_{1}u + b_{1} + w _{0}r^{m_{1}} \dfrac{u^{2}}{(r+du)^{m_{1}}}  x_{2}^{m_{2} } ,\\[2ex]
\dfrac{dx_{2} }{dt} &=-a_{2} x_{2} + w_{1}  \dfrac{r^{m_{1}}}{(r+du)^{m_{1}}}  x_{2}^{m_{2} },
\end{cases}
\end{equation}

Via comparison we have
\begin{eqnarray}
&& \frac{du }{dt} \nonumber \\
&=& -a_{1}u + b_{1} + w _{0}r^{m_{1}} \frac{u^{2}}{(r+du)^{m_{1}}}  x_{2}^{m_{2} }  \nonumber \\
&\leq& -a_{1}u + b_{1} +\frac{w _{0}r^{m_{1}}}{d^{m_{1}}} u^{2-m_{1}}x_{2}^{m_{2}}  \nonumber \\
&\leq& -a_{1}u + b_{1} +\frac{w _{0}r^{m_{1}}}{d^{m_{1}}} u^{2-m_{1}}K_{2}^{m_{2}}
 \end{eqnarray}
Recall that $x_{1}$ and $x_{2}$ are bounded above.  Let the upper bounds for $x_1$ and $x_2$ be constants $K_1$ and $K_2$, respectively. Now, we desire that $\frac{du }{dt} \leq 0$, $\forall t$, as this will ensure that $u$ cannot blow-up in finite time, or $x_{1}$ cannot go extinct in finite time.
lets compare to
\begin{equation}
\label{eq:n12}
 \frac{du }{dt}  =  -a_{1}u + b_{1} +\frac{w _{0}r^{m_{1}}}{d^{m_{1}}} u^{2-m_{1}}K_{2}^{m_{2}}
\end{equation}

let us set $u = \dfrac{b_{1}}{a_{1}} + v$, this changes \eqref{eq:n12} to

\begin{equation}
\label{eq:n121}
 \frac{dv }{dt}  =  -a_{1}v  + \frac{w _{0}r^{m_{1}}}{d^{m_{1}}} (\frac{b_{1}}{a_{1}} + v)^{2-m_{1}}K_{2}^{m_{2}}
\end{equation}

Now, for $\dfrac{du }{dt} = \dfrac{du }{dt}  \leq 0$ one requires that
\begin{equation}
-a_{1}v  + \frac{w _{0}r^{m_{1}}}{d^{m_{1}}} (\frac{b_{1}}{a_{1}} + v)^{2-m_{1}}K_{2}^{m_{2}} \leq 0
\end{equation}
This is possible if we choose $r$ such that,
  \begin{equation}
r <\left( \frac{ (a_1 d^{m_{1}}) v(0)}{w_{0} (\frac{b_{1}}{a_{1}} + v(0))^{2-m_{1}}K_{2}^{m_{2}}} \right)^{\frac{1}{m_{1}}}.
 \end{equation}

 This is seen simply by looking in the phase for the $v$ equation, and shows, $v$ cant blow up in finite time. Thus neither can $u = \frac{b_{1}}{a_{1}} + v$, thus making the finite time extinction of $x_{1}$ an impossibility.
\end{proof}

\subsection{Existence of Equilibria and  Stability Analysis}~\\
Similar to Section 3, we investigate and analyze the equilibrium solutions of our mathematical model with prey refuge.  Consider the steady state equations of \eqref{refuge}:
\begin{align}
a_{1} x_{1} -b_{1} x_{1}^{2} - w _{0} \left(\frac{rx_{1} }{rx_{1} +d} \right)^{m_{{\kern 1pt} 1} } x_{2}^{m_{{\kern 1pt} 2} } &= 0  \label{requilibrium1} \\
-a_{2} x_{2} + w_{1} \left(\frac{rx_{1} }{rx_{1} +d} \right)^{m_{{\kern 1pt} 1}} x_{2}^{m_{{\kern 1pt} 2}} &= 0   \label{requilibrium2}
\end{align}
The above equations \eqref{requilibrium1} and \eqref{requilibrium2}  have three types of non-negative equilibria:

\begin{enumerate}[label=(\roman*)]
 \item The trivial equilibrium $E_0 (0,0)$.
 \item The predator-free equilibrium $E_1 (a_1/b_1,0)$.
 \item The interior equilibrium $E_2(x_1^*, x_2^*)$ where $x_1^*$ and $x_2^*$ are related by

 $$x_2^*=\dfrac{w_1}{w_0 a_2}\left[a_1 x_1^* - b_1 x_1^{*2}  \right] .$$
 \end{enumerate}

The variational matrix $\bf{J_r^*}$ of the model \eqref{refuge} around the interior equilibrium
$E_2(x_1^*, x_2^*)$ is
\begin{align*}
 \bf{J_r^*}=  \begin{bmatrix}
      b_{11}&   b_{12}\\
      b_{21}  &  b_{22}
     \end{bmatrix}.
\end{align*}
where
\begin{eqnarray*}
b_{11}&=&   a_1 - 2b_1x_1^*- m_1 w_0 {x_2^*}^{m_2}\left[ \frac{r}{d+r x_1^*}-\frac{r^2 x_1^*}{(r x_1^*+d)^2}  \right]\left(\frac{r x_1^*}{r x_1^*+d} \right)^{m_1-1},\\
b_{12}&=&  - {m_2 w_0 {x_2^*}^{m_2-1}{(r x_1^*)}^{m_1} \over (r x_1^*+d)^{m_1}},\\
b_{21}&=&  m_1dw_1 r^{m_1} {x_2^*}^{m_2}\bigg({{x_1^*}^{m_1-1}\over (rx_1^*+d)^{m_1+1}}\bigg) ,\\
b_{22}&=& -a_2+ m_2 w_1{ x_2^*}^{m_2-1}\bigg( {rx_1^*\over {rx_1^*+d}}\bigg)^{m_1}.
\end{eqnarray*}

The characteristic equation corresponding to ${\bf J_r^*}$ is given by
 $$\lambda^2 - \operatorname{tr} \,({\bf{J_r^*}}) \lambda + \det \,({\bf{J_r^*}}) =0,$$
where
\begin{eqnarray*}
\operatorname{tr} \,({\bf{J_r^*}}) &=& b_{11}+b_{22}\\
& = &  a_1  - a_2 -  2b_1x_1^*- m_1 w_0 {x_2^*}^{m_2}\left[ \frac{r}{d+r x_1^*}-\frac{r^2 x_1^*}{(r x_1^*+d)^2}  \right]\left(\frac{r x_1^*}{r x_1^*+d} \right)^{m_1-1} \\ & & +~ m_2 w_1{ x_2^*}^{m_2-1}\bigg( {rx_1^*\over {rx_1^*+d}}\bigg)^{m_1},
\end{eqnarray*}
and
\begin{eqnarray*}
\det \,({\bf{J_r^*}})&=& b_{11} b_{22}- b_{12} b_{21} \\
 &=& \left( a_1 - 2b_1x_1^*- m_1 w_0 {x_2^*}^{m_2}\left[ \frac{r}{d+r x_1^*}-\frac{r^2 x_1^*}{(r x_1^*+d)^2}  \right]\left(\frac{r x_1^*}{r x_1^*+d} \right)^{m_1-1} \right) \\
&& \times \left( -a_2+ m_2 w_1{ x_2^*}^{m_2-1}\bigg( {rx_1^*\over {rx_1^*+d}}\bigg)^{m_1} \right) \\
&& - \left(  - {m_2 w_0 {x_2^*}^{m_2-1}{(r x_1^*)}^{m_1} \over (r x_1^*+d)^{m_1}} \right)    \left(  m_1dw_1 r^{m_1} {x_2^*}^{m_2}\bigg({{x_1^*}^{m_1-1}\over (rx_1^*+d)^{m_1+1}}\bigg)   \right).
\end{eqnarray*}
Here, $\operatorname{tr} \,({\bf{J_r^*}})$ and $\det \,({\bf{J_r^*}}) $ represents the trace and determinant of the variational matrix.  Hence the stability of $E_2(x_1^*, x_2^*)$ is determined by the sign of $\det \,({\bf{J_r^*}}) $ and $\operatorname{tr} \,({\bf{J_r^*}})$.

The above results are summarized in the following theorem.

\begin{theorem}\label{localstabintrefuge}
The interior equilibrium $E_2(x_1^*, x_2^*)$ of system \eqref{refuge} is locally asymptotically stable if  $\operatorname{tr} \,({\bf{J_r^*}}) < 0$ and $\det \,({\bf{J_r^*}})  > 0$ by Routh-Hurwitz stability criteria.
\end{theorem}

\begin{proof}
The proof follows directly from the above discussion and hence omitted for brevity.
\end{proof}

\subsection{Bifurcation Analysis}~\\

In this subsection, we analyze the qualitative changes in the dynamical behavior of model \eqref{refuge} under the effect of varying a specific parameter. The conditions and restrictions for the occurrence of saddle-node, Hopf, and transcritical bifurcations are derived. The classification are of codimension one bifurcations.

\subsubsection{Saddle-Node Bifurcation}

We investigate the possibility of saddle-node bifurcation of the positive interior equilibrium $E_2$ by using the intrinsic death rate of the predator population as a bifurcation parameter.

The following theorem states the restrictions for occurrence of a saddle-node bifurcation for model \eqref{refuge}.

\begin{theorem}\label{saddle_node_refuge_a2}
The model \eqref{refuge} undergoes a saddle-node bifurcation around $E_2$ at $a_2^*$ when the system parameters satisfy the restriction $\det \,({\bf{J_r^*}}) = 0 $ along with the condition  $\operatorname{tr} \,({\bf{J_r^*}})<0$.
\end{theorem}

\begin{proof}
To validate the restriction for the occurrence of saddle-node bifurcation, we apply Sotomayor's theorem \cite{Perko13} at $a_2=a_2^*$. At $a_2=a_2^*$, it can be seen that $\det \,({\bf{J_r^*}}) = 0 $ and $\operatorname{tr} \,({\bf{J_r^*}})<0$ which indicates that the Jacobian $\,({\bf{J_r^*}})$ admits a zero eigenvalue. Let $U$ and $V$ be the eigenvectors corresponding to the zero eigenvalue of the matrix $\,({\bf{J_r^*}})$ and $\,({\bf{J_r^*}})^T$ respectively. We obtain that $U=(u_1,u_2)^T$ and $V=(v_1,v_2)^T$, where $u_1=-\frac{b^*_{12}u_2}{b^*_{11}}$, $v_1=-\frac{b^*_{21}v_2}{b^*_{11}}$ and $u_2,v_2 \in \mathbb{R} \setminus \{0\}$.

Let $G=(G_1,G_2)^T$ and $X=(x_1^*,x_2^*)^T$, where $G_1, G_2$ are given by

\begin{align}
G_1 &= a_{1} x_{1} -b_{1} x_{1}^{2} - w _{0} \left(\frac{rx_{1} }{rx_{1} +d} \right)^{m_{{\kern 1pt} 1} } x_{2}^{m_{{\kern 1pt} 2} }, \label{SNB_G1}  \\
G_2 &= -a_{2} x_{2} + w_{1} \left(\frac{rx_{1} }{rx_{1} +d} \right)^{m_{{\kern 1pt} {\kern 1pt} 1} } x_{2}^{m_{{\kern 1pt} {\kern 1pt} 2} }.\label{SNB_G2}
\end{align}
Now
\begin{align*}
V^T G_{a_2}(X,a_2^*) = (v_1,v_2)(0,-x_2)^T=-v_2 x_2 \neq 0,
\end{align*}
and
\begin{align*}
V^T \left[D^2 G(X,a_2^*)(U,U)\right] \neq 0.
\end{align*}
Hence from Sotomayor's theorem the model \eqref{refuge} undergoes a saddle-node bifurcation around $E_2$ at $a_2=a_2^*$.
\end{proof}

\begin{theorem}\label{saddle_node_refuge_w1}
The model \eqref{refuge} undergoes a saddle-node bifurcation around $E_2$ at $w_1^*$ when the system parameters satisfy the restriction $\det \,({\bf{J_r^*}}) = 0 $ along with the condition  $\operatorname{tr} \,({\bf{J_r^*}})<0$.
\end{theorem}

\begin{theorem}\label{saddle_node_refuge_w0}
The model \eqref{refuge} undergoes a saddle-node bifurcation around $E_2$ at $w_0^*$ when the system parameters satisfy the restriction $\det \,({\bf{J_r^*}}) = 0 $ along with the condition  $\operatorname{tr} \,({\bf{J_r^*}})<0$.
\end{theorem}

\begin{theorem}\label{saddle_node_refuge_b1}
The model \eqref{refuge} undergoes a saddle-node bifurcation around $E_2$ at $b_1^*$ when the system parameters satisfy the restriction $\det \,({\bf{J_r^*}}) = 0 $ along with the condition  $\operatorname{tr} \,({\bf{J_r^*}})<0$.
\end{theorem}

\begin{theorem}\label{saddle_node_refuge_a1}
The model \eqref{refuge} undergoes a saddle-node bifurcation around $E_2$ at $a_1^*$ when the system parameters satisfy the restriction $\det \,({\bf{J_r^*}}) = 0 $ along with the condition  $\operatorname{tr} \,({\bf{J_r^*}})<0$.
\end{theorem}

\begin{proof}
The proof of Theorem \ref{saddle_node_refuge_w1}, Theorem \ref{saddle_node_refuge_w0}, Theorem \ref{saddle_node_refuge_b1} and Theorem \ref{saddle_node_refuge_a1}   are similar to proof in Theorem \ref{saddle_node_refuge_a2} and omitted for brevity.
\end{proof}

\subsubsection{Transcritical Bifurcation}

Here, we investigate the possibility of the existence of a transcritical bifurcation for the model \eqref{refuge}. Transcritical bifurcation occurs when an equilibrium point interchanges its stability when it collides with another equilibrium point as a parameter is varied. The prey refuge parameter $r$ is used as a bifurcation parameter.

The variational matrix ${\bf{J_{r_1^{*}}}}$ of the model \eqref{refuge} for $0<m_1<1$ and $m_2=1$ evaluated at
 $$r=r_1^{*}=\dfrac{b_1 d}{a_1}\left(\dfrac{a_1^{\frac{1}{m_1}}}{w_1^{\frac{1}{m_1}}-a_2^{\frac{1}{m_1}}}  \right)$$
 around the predator-free equilibrium $E_1(a_1/b_1,0)$ is given by

\begin{align*}
 \bf{J_{r_1^{*}}}=  \begin{bmatrix}
      -a_1 &   -w_0\left(\dfrac{r_1^* a_1}{r_1^* a_1 + b_1 d}   \right)^{m_1}\\
      0  &  0
     \end{bmatrix}.
\end{align*}
At $r=r_1^{*}$, the matrix ${\bf{J_{r_1^{*}}}}$ has a negative eigenvalue and a zero eigenvalue. Let $U$ and $V$ be the eigenvectors corresponding to the zero eigenvalue of the matrix $\,({\bf{J_{_1r^{*}}}})$ and $\,({\bf{J_{r_1^{*}}}})^T$ respectively. Then \\
$$ U=\left( 1, -\dfrac{a_1}{w_0}\left( 1+\dfrac{b_1 d}{r_1^* a_1} \right)^{m_1} \right)^T,  \qquad V=(0,1)^T. $$
Let $G=(G_1,G_2)^T$ and $X=(a_1/b_1,0)^T$, where $G_1, G_2$ are defined in \eqref{SNB_G1} and \eqref{SNB_G2}.
Now we have
\begin{align*}
V^T G_{r}(X,r_1^{*}) = (0,1)(0,0)^T= 0,
\end{align*}
additionally
\begin{align*}
V^T \left[D G_{r}(X,r_1^{*})U\right] \neq 0
\end{align*}
and
\begin{align*}
V^T \left[D^2 G(X,r_1^{*})(U,U)\right] \neq 0.
\end{align*}
Hence using Sotomayor's theorem the model \eqref{refuge} undergoes a transcritical  bifurcation around $E_1$ when the refuge $r$ crosses the critical value of the parameter  $r_1^{*}$.

The  above results are summarized in the following theorem.

\begin{theorem}\label{transcritical_bifurcation}
The model \eqref{refuge} undergoes a transcritical bifurcation around $E_1(a_1/b_1,0)$ when the refuge $r$ crosses the critical value of parameter $r_1^{*}$, where \\
$r_1^{*}=\dfrac{b_1 d}{a_1}\left(\dfrac{a_1^{\frac{1}{m_1}}}{w_1^{\frac{1}{m_1}}-a_2^{\frac{1}{m_1}}}  \right)$.
\end{theorem}~\

\subsubsection{Hopf-Bifurcation}

We investigate the possibility of Hopf-bifurcation of the positive interior equilibrium $E_2$ by using the per capita rate of self-reproduction for the prey, $a_1$ as a bifurcation parameter. Then, the characteristic equation corresponding to model \eqref{refuge} at $E_2$ is given by
\begin{align}\label{character_bifur_refuge}
\lambda^2 + A_1(a_1)\lambda + B_1(a_1) = 0,
\end{align}
where
$A_1 = - \operatorname{tr} \,({\bf{J_r^*}}) = - (b_{11} + b_{22})$ and
$B_1 =  \det \,({\bf{J_r^*}}) =  b_{11}  b_{22}- b_{12}b_{21}$.

The instability of model \eqref{refuge} is demonstrated via the following theorem by considering $a_1$ as a bifurcation parameter.

\begin{theorem}\label{hopf-bifurcationrefuge_a1}
The model \eqref{refuge} undergoes a Hopf-bifurcation around $E_2(x_1^*,x_2^*)$ when $a_1$ crosses some critical value of parameter  $a_1^*$, where
\begin{align*}
a_1^*  =  a_2 +  2b_1x_1^* + m_1 w_0 {x_2^*}^{m_2}\left[ \frac{r}{d+r x_1^*}-\frac{r^2 x_1^*}{(r x_1^*+d)^2}  \right]\left(\frac{r x_1^*}{r x_1^*+d} \right)^{m_1-1} \\
- m_2 w_1{ x_2^*}^{m_2-1}\bigg( {rx_1^*\over {rx_1^*+d}}\bigg)^{m_1},
\end{align*}
provided
\begin{enumerate}[label=(\roman*)]
\item $A_1(a_1)=0 $,
\item $B_1(a_1)>0$,
\item $\dfrac{d}{da_1}\left. Re \lambda_i (a_1) \right|_{a_1=a_1^*}\neq 0$ at $a_1=a^*_1,~ i=1,2$.
\end{enumerate}
\end{theorem}

\begin{proof}
Clearly $A_1(a_1)$ and $B_1(a_1)$ are the smooth functions of $a_1$. The roots of the equation \eqref{character_bifur_refuge} are of the form $\lambda_1=\vartheta (a_1) + i \varpi (a_1)$ and $\lambda_2=\vartheta (a_1) - i \varpi (a_1)$ where $\vartheta (a_1)$ and $\varpi (a_1)$ are real functions.

At $a_1=a_1^*$, the characteristic equation \eqref{character_bifur_refuge} reduces to
 \begin{align}\label{charac_reduced_bifur_refuge}
 \lambda^2 + B_1(a_1)=0
 \end{align}
By solving for the roots of equation \eqref{charac_reduced_bifur_refuge}, we obtain $\lambda_1=i\sqrt{B_1}$ and  $\lambda_2=-i\sqrt{B_1}$. Therefore, we have purely imaginary eigenvalues. Hence, we are left with validating the transversality condition. Namely,
$$\dfrac{d}{da_1}Re \lambda_i (a_1) |_{a_1=a_1^*}\neq 0, i=1,2.$$
Substituting $\lambda (a_1) = \vartheta (a_1) + i \varpi (a_1)$ into equation \eqref{character_bifur_refuge}, we obtain
 \begin{align}
 (\vartheta (a_1) + i \varpi (a_1))^2 + A_1(a_1) (\vartheta (a_1) + i \varpi (a_1)) + B_1(a_1)=0.
 \end{align}
Upon taking the derivative with respect to $a_1$ we obtain:
 \begin{align*}
 2(\vartheta (a_1) + i \varpi (a_1)) (\dot{\vartheta} (a_1) + i \dot{\varpi} (a_1)) + A_1(a_1) (\dot{\vartheta} (a_1) + i \dot{\varpi} (a_1))  \\
 + \dot{A_1} (a_1) (\vartheta (a_1) + i \varpi (a_1)) + \dot{B_1} (a_1) = 0.
 \end{align*}
Separating the real and imaginary parts, we have
 \begin{align*}
 \dot{\vartheta}(a_1) (2 \vartheta (a_1) + A_1 (a_1)) + \dot{\varpi (a_1)} (-2 \varpi (a_1)) + \dot{A_1} (a_1) \vartheta (a_1) + \dot{B_1} (a_1) = 0,
 \end{align*}
which implies
  \begin{align}\label{hopf_cond1}
 \dot{\vartheta} (a_1) Z_1 (a_1) - \dot{\varpi} (a_1) Z_2 (a_1) + Z_3 (a_1) = 0,
 \end{align}
and
  \begin{align*}
 \dot{\vartheta}(a_1) (2 \varpi (a_1)) + \dot{\varpi} (a_1) (2 \vartheta (a_1) + A_1 (a_1)) + \dot{A_1} (a_1) \varpi (a_1) =0,
 \end{align*}
which implies
  \begin{align}\label{hopf_cond2}
 \dot{\vartheta} (a_1) Z_2 (a_1) + \dot{\varpi} (a_1) Z_1 (a_1) + Z_4 (a_1) = 0,
 \end{align}
where $Z_1 (a_1)= 2 \vartheta (a_1) + A_1 (a_1)$, $Z_2 (a_1)= 2 \varpi (a_1)$, $Z_3 (a_1)=  \dot{A_1} (a_1) \vartheta (a_1) + \dot{B_1} (a_1) $ and $Z_4 (a_1)=\dot{A_1} (a_1) \varpi (a_1) $.
Multiplying equation \eqref{hopf_cond1} by $Z_1 (a_1)$ and equation \eqref{hopf_cond2} by $Z_2 (a_1)$ and then adding them, we obtain
 \begin{align}\label{hopf_cond3}
 (Z_1^2 (a_1) + Z_2^2 (a_1)) \dot{\vartheta} (a_1) + Z_1 (a_1) Z_3 (a_1) + Z_2 (a_1) Z_4 (a_1) =0,
 \end{align}
thus solving for $\dot{\vartheta} (a_1)$ from equation \eqref{hopf_cond3} and at $a_1=a_1^*$,
\begin{align*}
\dfrac{d}{da_1}Re \lambda_i (a_1) |_{a_1=a_1^*} = \dot{\vartheta} (a_1^*)=\dfrac{-\left[ Z_1 (a_1^*) Z_3 (a_1^*) + Z_2 (a_1^*) Z_4 (a_1^*) \right]}{Z_1^2 (a_1^*) + Z_2^2 (a_1^*)}.
 \end{align*}

It is easy to verify that $Z_1 (a_1^*) Z_3 (a_1^*) + Z_2 (a_1^*) Z_4 (a_1^*) \neq 0$ and $Z_1^2 (a_1^*) + Z_2^2 (a_1^*) \neq 0$ which implies  $\frac{d}{da_1}Re \lambda_i (a_1) |_{a_1=a_1^*} \neq 0$. Hence, a Hopf-bifurcation occurs around $E_2(x_1^*,x_2^*)$ at $a_1=a_1^*$.

\end{proof}

\begin{theorem}\label{hopf-bifurcationrefuge_r}
The model \eqref{refuge} undergoes a Hopf-bifurcation around $E_2(x_1^*,x_2^*)$ when the refuge $r$ crosses some critical value of parameter  $r_1^{**}$ provided
\begin{enumerate}[label=(\roman*)]
\item $A_1(r)=0 $,
\item $B_1(r)>0$,
\item $\dfrac{d}{dr}\left. Re \lambda_i (r) \right|_{r=r_1^{**}}\neq 0$ at $r=r^{**}_1,~ i=1,2$.
\end{enumerate}
\end{theorem}

\begin{proof}
The proof of Theorem \ref{hopf-bifurcationrefuge_r}  is similar to proof in Theorem \ref{hopf-bifurcationrefuge_a1} and omitted for brevity.
\end{proof}


\section{Numerical Simulations}\label{section:Numerical_simulations}
Numerical simulations of model \eqref{EquationMain} are performed in this section to correlate with some of  our key analytical findings. The numerical simulations and figures have been developed using MATLAB$\circledR$ R2019b, MATCONT\cite{Matcont}, and XPPAUT\cite{Ermentrout02}. For convenience, the parameters used in simulations are given in Table \ref{table:Paramset2}.

\begin{table}[h!]
  \begin{center}
    \caption{Parameters used in the simulations of Figs. \ref{fig:nullcline_manifold1}, \ref{fig:nullcline_manifold2}, \ref{fig:bifurcation1}. \ref{fig:time_phase1}, \ref{fig:bifurcation2}, and \ref{fig:TB_Hopf_bifurcation_refuge}.}
    \label{table:Paramset2}
    \begin{tabular}{@{}l l@{}l l@{}l l@{}l @{} l@{}}
      \midrule
&$a_1$                 = $0.6$ \quad \quad \quad
&$a_2$                 = $1$ \quad \quad \quad
&$b_1$                 = $0.063$ \quad \quad \quad
&$w_0$              = $1$ \quad \quad \quad
&$d$              = $2$  \quad \quad \quad \\
&$w_1$                     = $2$ \quad \quad \quad
&$m_1$                     = $0.8$  \quad \quad \quad
&$m_2$                     = $1$  \\
      \bottomrule
    \end{tabular}
  \end{center}
\end{table}
%
For the parameter values in Table \ref{table:Paramset2}, the predator-free equilibrium point $E_1(9.52381,0)$ is a saddle and $E_2(1.45094,1.47587)$ is a repeller, see Fig. \ref{fig:nullcline_manifold1}(a). In  Fig. \ref{fig:nullcline_manifold1}(a), $W^{s}(E_0)$ is above $W^{u}(E_1)$ and  $E_2(1.45094,1.47587)$ is surrounded by a limit cycle. It is seen that in Fig. \ref{fig:nullcline_manifold1}(b), when $a_1=2$, $b_1=0.21$ and all the other parameter sets are given in Table \ref{table:Paramset2}, then $W^{s}(E_0)$ is under $W^{u}(E_1)$, where $E_1(9.52381,0)$ and $E_2(1.45094,4.91957)$. So the predator-free equilibrium point $E_1(9.52381,0)$ is a saddle and $E_2(1.45094,4.91957)$ is a repeller. Here all positive solutions converge towards $E_0$. In  Fig. \ref{fig:nullcline_manifold2}, when $a_1=1.5$, $b_1=0.18$, $a_2=1.4$ and all the other parameter sets are given in Table \ref{table:Paramset2}, we observe that $E_2(3.55994,4.36963)$ is an attractor and $E_1(8.33333,0)$ is a saddle. Also $W^{s}(E_0)$ is above $W^{u}(E_1)$ and the model \eqref{EquationMain} has a unique limit cylce whose basin of attraction is the region under $W^s(E_0)$.
However, in n Fig. \ref{fig:nullcline_manifold2}, the numerical simulations illustrate that $E_2$ is not globally asymptotically stable for the given parameter set.

\begin{figure}[!htb]
\begin{center}
\subfigure[]{
    \includegraphics[scale=.13]{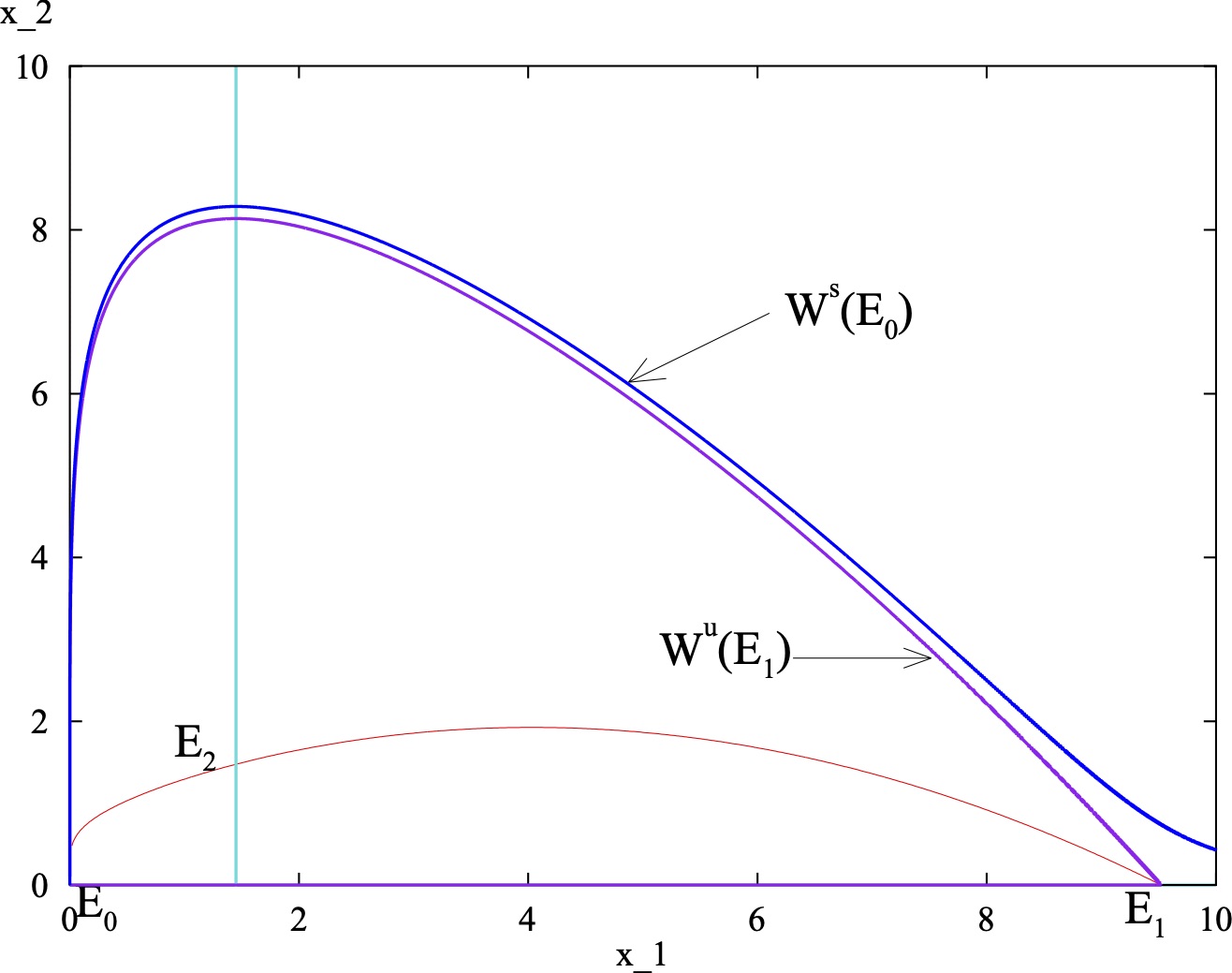}}
\subfigure[]{
    \includegraphics[scale=.13]{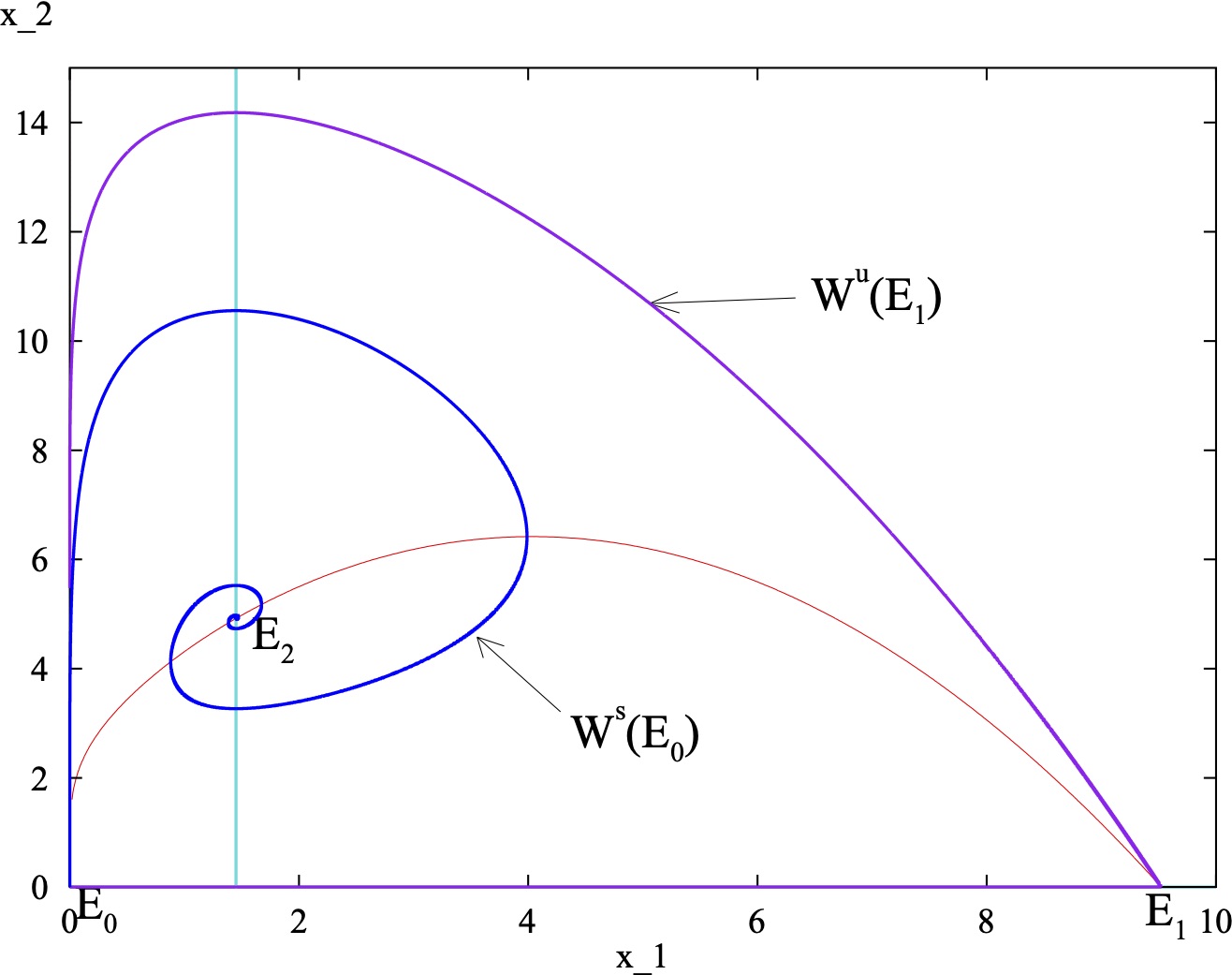}}
\end{center}
 \caption{The predator and prey nullclines for model \eqref{EquationMain} are represented by turquoise and red respectively. (a) $W^{s}(E_0)$ is above $W^{u}(E_1)$: $E_2$ is unstable (b) $W^{u}(E_1)$ is above $W^{s}(E_0)$, here $a_1=2$ and $b_1=0.21$: $E_2$ is unstable.  Other parameter sets are given in Table \ref{table:Paramset2}. }
      \label{fig:nullcline_manifold1}
\end{figure}
\begin{figure}[!htb]
\begin{center}
    \includegraphics[scale=.14]{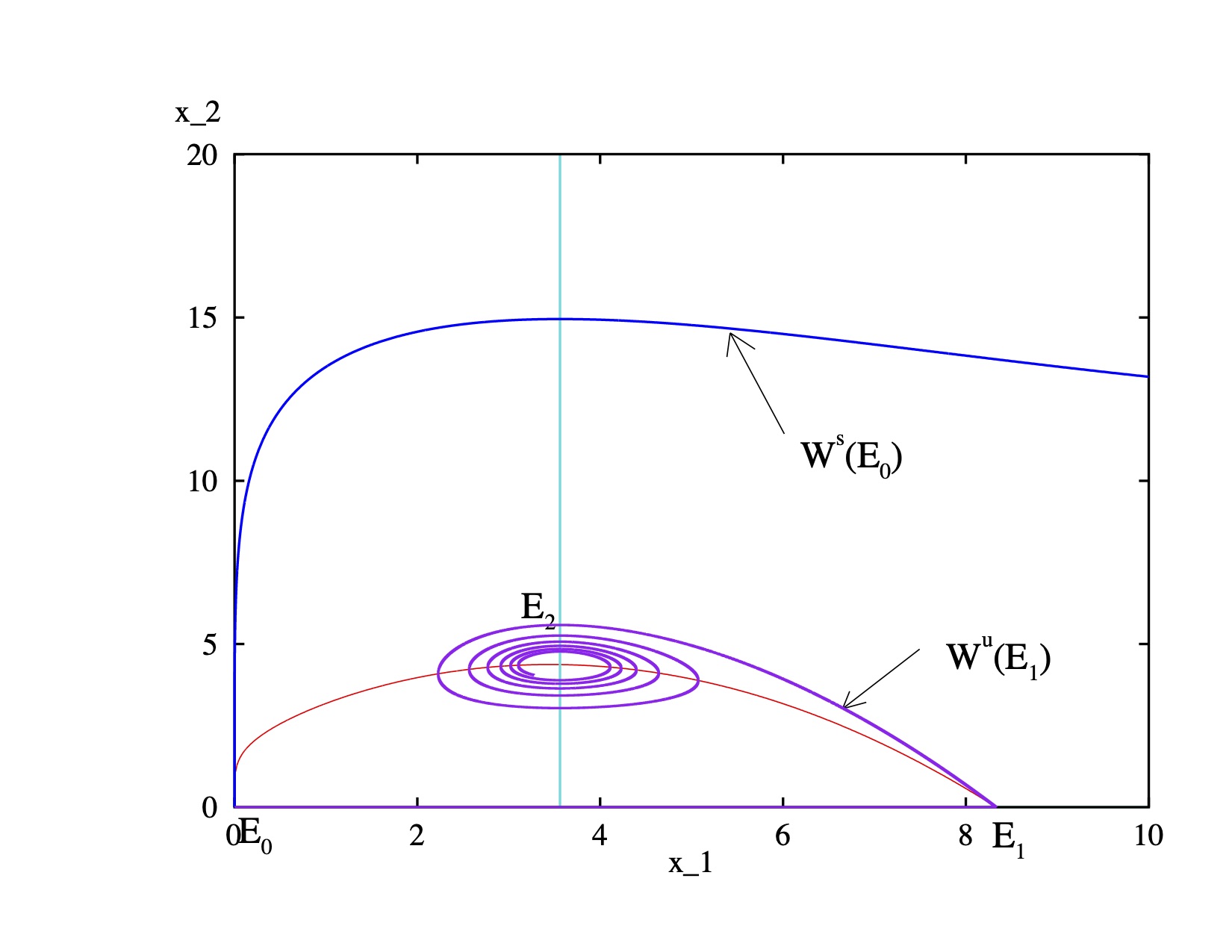}
\end{center}
 \caption{ $E_2$ is an attractor and $W^{s}(E_0)$ is above $W^{u}(E_1)$. Here $a_1=1.5$, $b_1=0.18$ and $a_2=1.4$. Other parameter sets are given in Table \ref{table:Paramset2}.  }
      \label{fig:nullcline_manifold2}
\end{figure}
\begin{figure}[H]
\begin{center}
   \includegraphics[scale=.13]{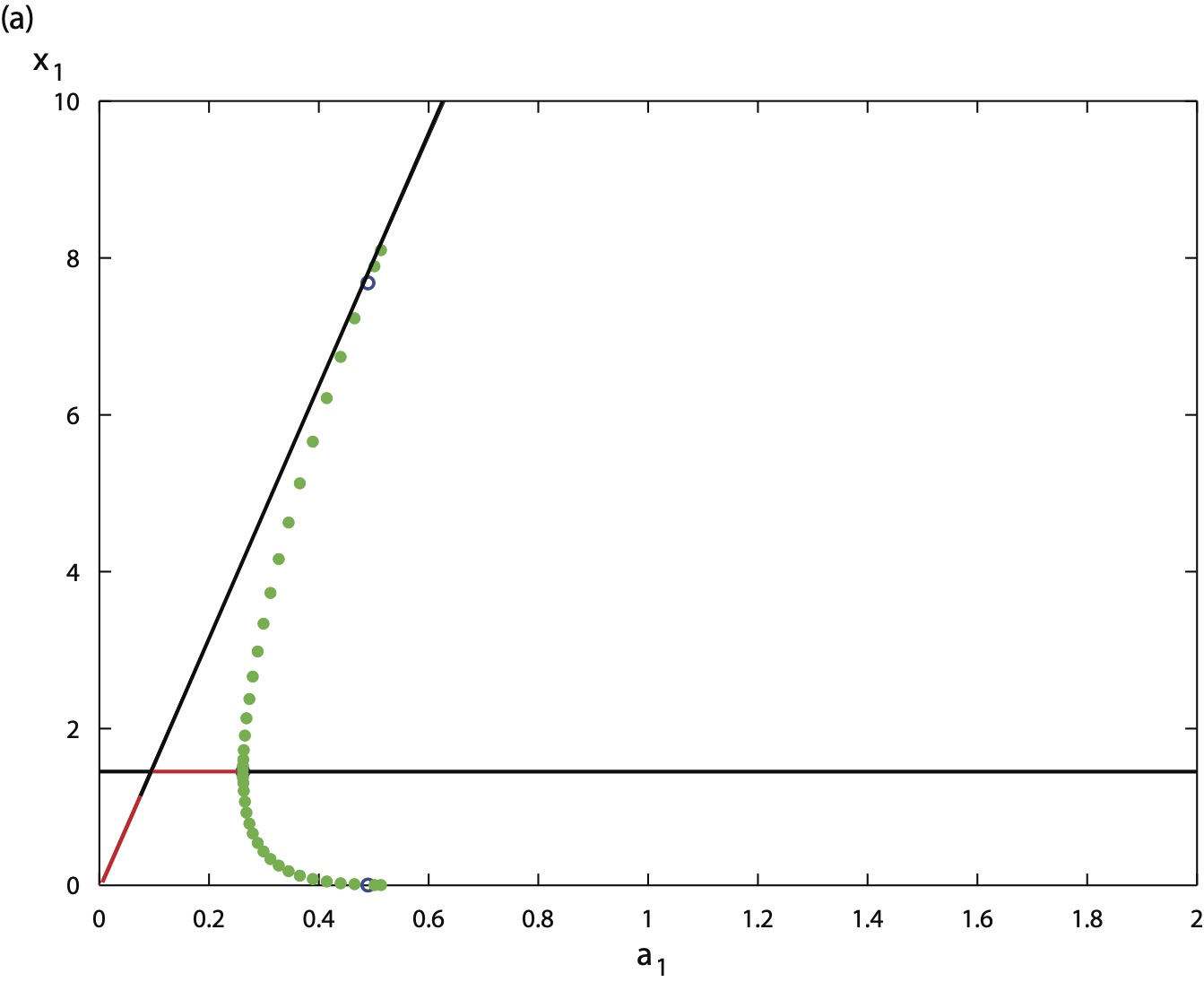}
    \includegraphics[scale=.13]{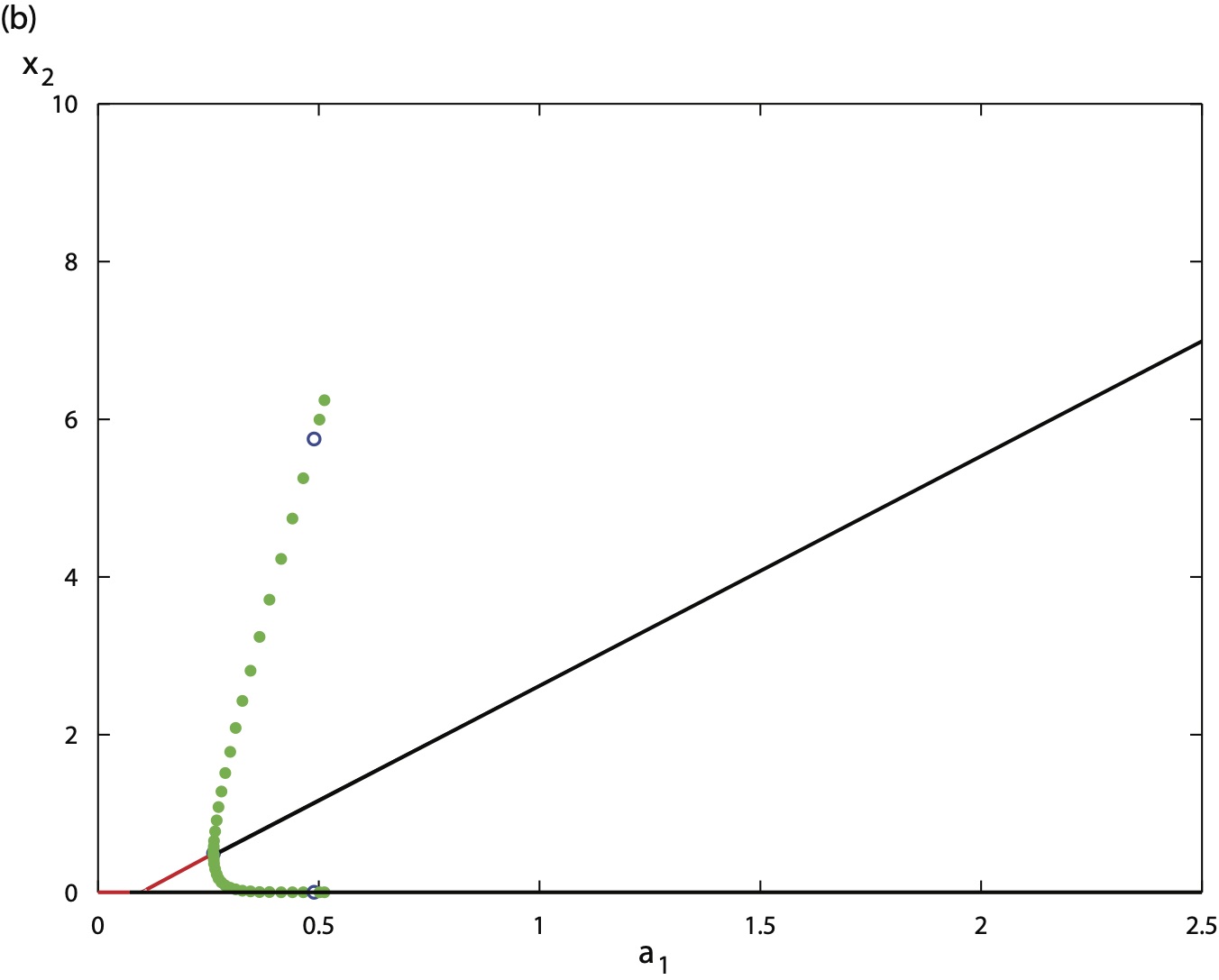}
\end{center}
 \caption{Bifurcation diagrams of the model \eqref{EquationMain}, as $a_1$ changes. The stable and unstable interior equilibriums are given by the lines in red and black, respectively. The solid circles (green) represent stable limit cycles and the open circles (blue) represent unstable
limit cycles. (a) prey ($x_1$) (b) predator ($x_2$). Parameter set are given in Table \eqref{table:Paramset2}. }
      \label{fig:bifurcation1}
\end{figure}
Furthermore, for the parameter sets in Table \ref{table:Paramset2}, we employ AUTO as implemented in the continuation software XPPAUT to analyze the bifurcation diagrams of the model \eqref{EquationMain} in  Fig. \ref{fig:bifurcation1}. The model undergoes Hopf-bifurcation around $E_2(1.45094,0.49456)$ as the parameter  $a_1$ crosses its critical value $a_1^*=0.261835$. The branch of periodic orbits emitting from $a_1^*$ are stable and the first Lyapunov coeffiicient \cite{Perko13}, $\sigma=-1.49929e^{-2}<0$ (obtained with the aid of MATCONT), hence the Hopf-bifurcation is supercritical.
%
\begin{table}[h!]
  \begin{center}
    \caption{Parameters used in the simulations of Figs. \ref{fig:time_series_SN}, \ref{fig:SN_Main}, \ref{fig:time_series_refuge_SN}, and \ref{fig:SN_refuge}.}
    \label{table:Paramset3}
    \begin{tabular}{@{}l l@{}l l@{}l l@{}l @{} l@{}}
      \midrule
&$a_1$                 = $0.5$ \quad \quad \quad
&$a_2$                 = $0.7$ \quad \quad \quad
&$b_1$                 = $0.05$ \quad \quad \quad
&$w_0$              = $0.2$ \quad \quad \quad
&$d$              = $0.2$  \quad \quad \quad \\
&$w_1$                     = $4$ \quad \quad \quad
&$m_1$                     = $0.5$  \quad \quad \quad
&$m_2$                     = $0.5$  \\
      \bottomrule
    \end{tabular}
  \end{center}
\end{table}
\begin{figure}[H]
\begin{center}
    \includegraphics[scale=.27]{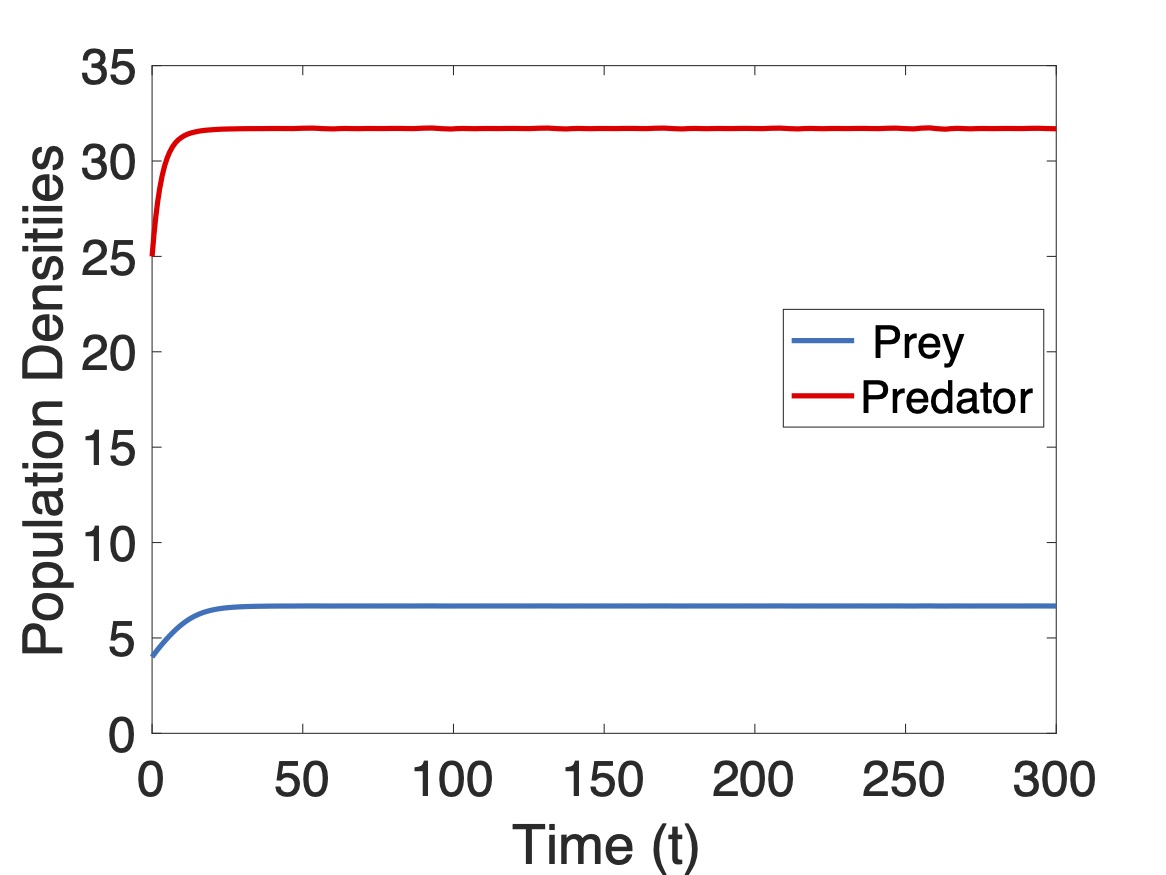}
\end{center}
 \caption{Time series depicting the stability behavior of the interior equilibrium point $(6.67563,31.7032)$ for the  parameter sets given in Table \ref{table:Paramset3}.  }
      \label{fig:time_series_SN}
\end{figure}
%
\begin{figure}[!htb]
\begin{center}
    \subfigure[]{
    \includegraphics[scale=.095]{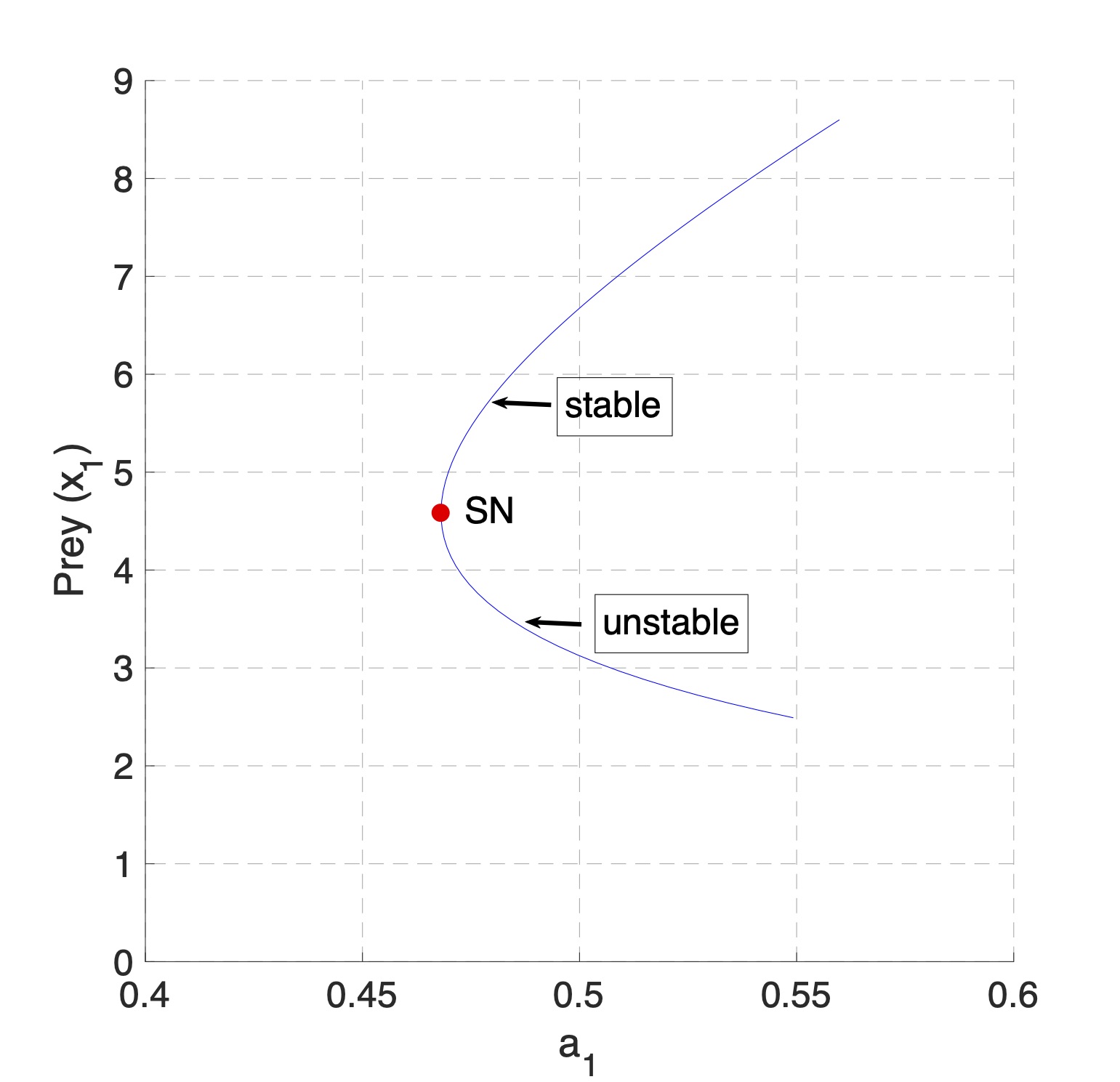}}
    \subfigure[]{
    \includegraphics[scale=.095]{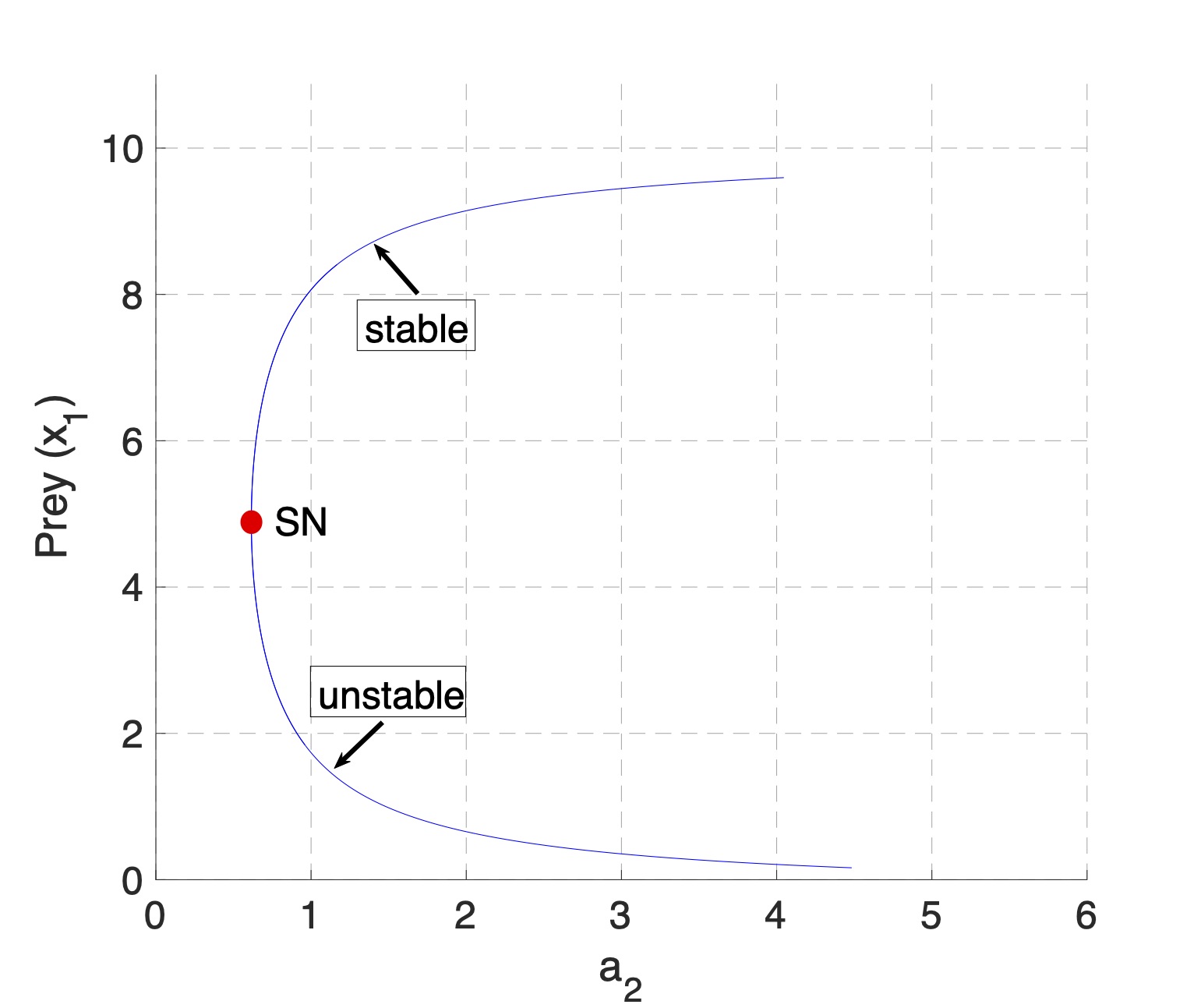}}
    \subfigure[]{
    \includegraphics[scale=.090]{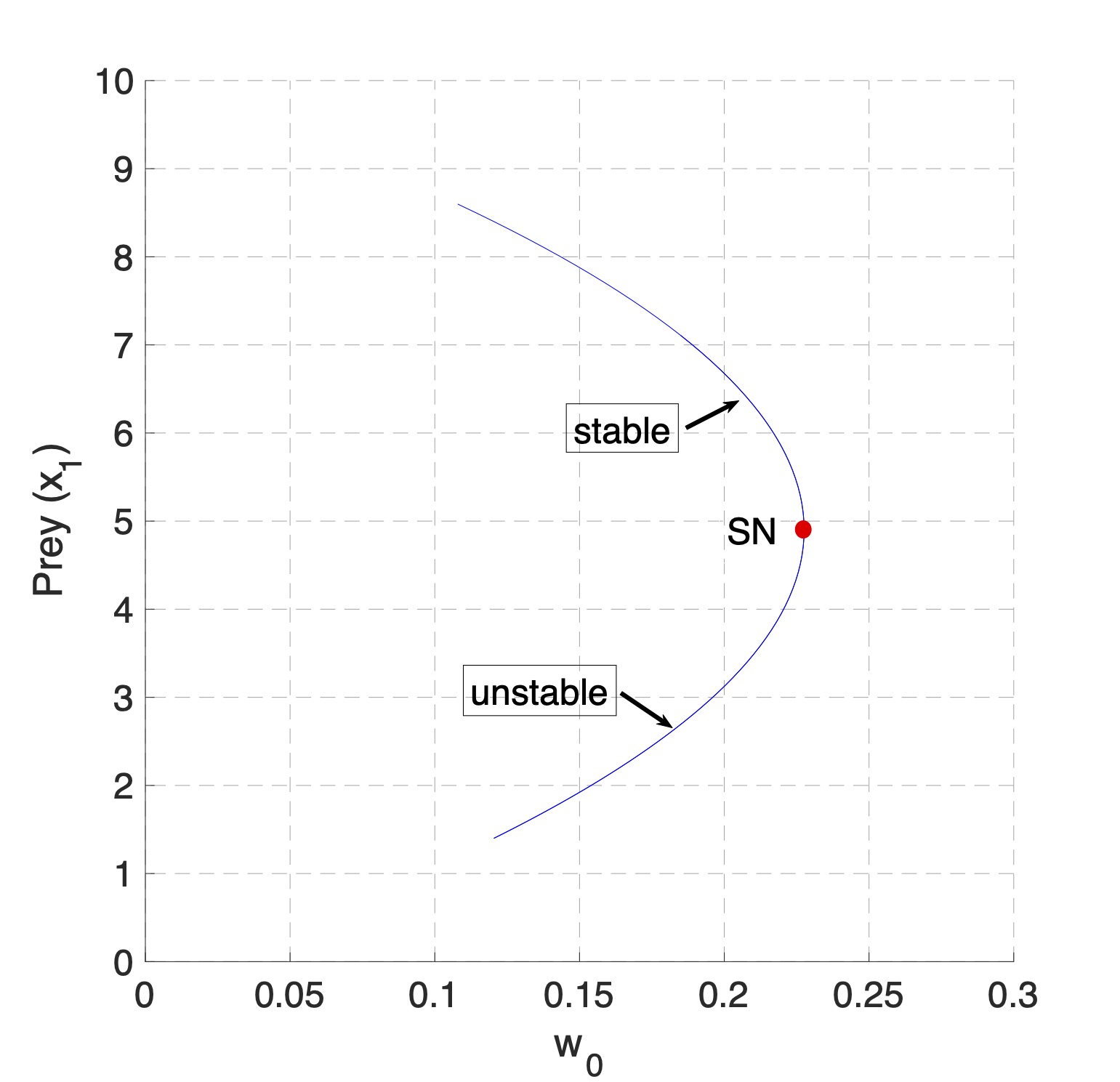}}
   \subfigure[]{
    \includegraphics[scale=.085]{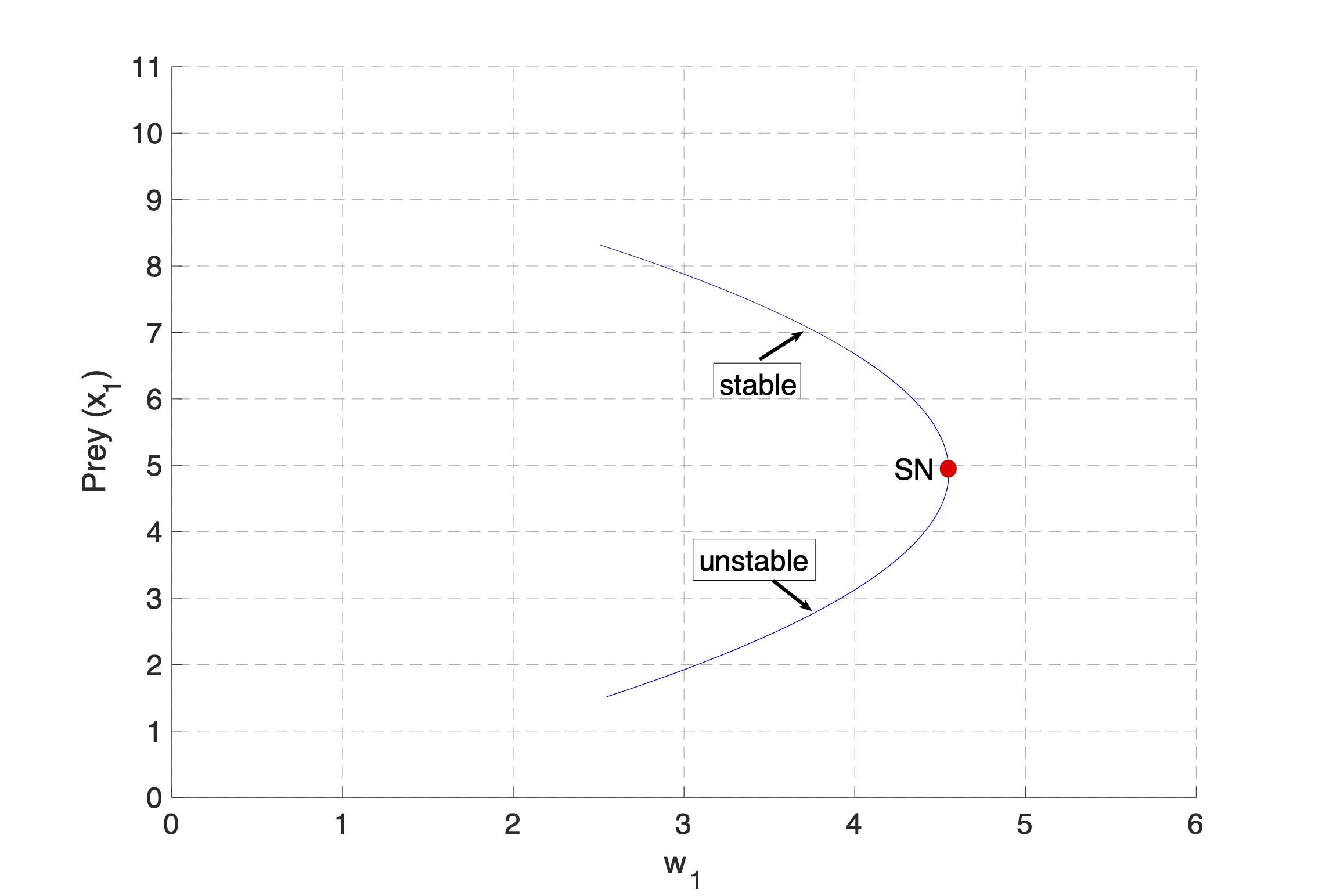}}
  \subfigure[]{
    \includegraphics[scale=.095]{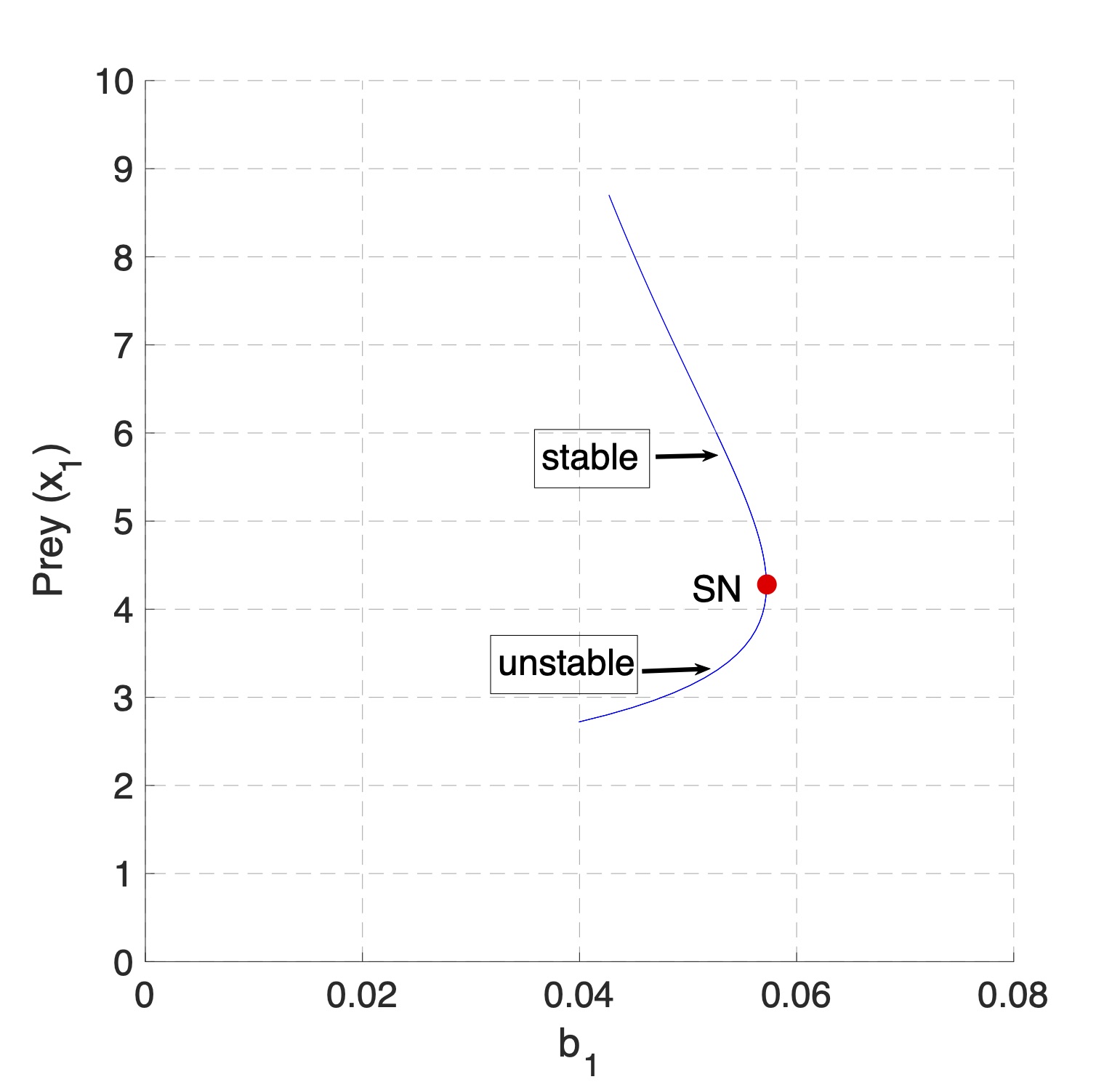}}
\end{center}
 \caption{Bifurcation diagrams illustrating (a) SN at $a_1=a_1^*=0.46809$, (b) SN at $a_2=a_2^*=0.61515$, (c) SN at $w_0=w_0^*=0.22759$, (d) SN at $w_1=w_1^*=4.55175$,  (e) SN at $b_1=b_1^*=0.05722$. Other parameter sets are given in Table \ref{table:Paramset3}. (SN: Saddle-node bifurcation.)  }
      \label{fig:SN_Main}
\end{figure}

Also, for the parameter sets in Table \ref{table:Paramset3} of model \eqref{EquationMain}, we obtain the following two interior equilibrium points $E_2^1(3.12437,30.6886)$ and $E_2^2(6.67563,31.7032)$ and the predator-free equilibrium point is $E_1(10,0)$. The eigenvalues associated with $E_2^1(3.12437,30.6886)$ are $-0.343043$ and $0.170265$, hence $E_2^1$ is a saddle. The eigenvalues associated with $E_2^2(6.67563,31.7032)$ are $-0.34517$ and $-0.17481$, hence $E_2^2$ is locally asymptotically stable, see Fig. \ref{fig:time_series_SN}. We note here that, $E_1(10,0)$ cannot be analyzed using the linear stability method since $m_1=m_2=0.5<1$. We observed that the model \eqref{EquationMain} undergoes a saddle-node bifurcation around $E_2(x_1^*,x_2^*)$ when the following bifurcation parameters $a_1,a_2,w_0,w_1$ and $b_1$ crosses their corresponding critical values $a_1^*=0.46809,a_2^*=0.61515,w_0^*=0.22759,w_1^*=4.55175,$ and $b_1^*=0.05722$ respectively. The saddle-node bifurcation diagrams are depicted in Fig. \ref{fig:SN_Main}

~\\
Now, we perform numerical simulations of model \eqref{refuge} to verify  some of our analytical results. For $r=0.3$ and all other parameter values given in Table \ref{table:Paramset2}, the predator-free equilibrium point $E_1(9.52381,0)$ is a saddle and $E_2(4.83648,9.52147)$ is an attractor (stable), see Fig. \ref{fig:time_phase1}.
By introducing a prey refuge of $r=0.3$, we observed in Fig. \ref{fig:time_phase1}(c) that $W^{s}(E_0)$ is above $W^{u}(E_1)$ as compared to Fig. \ref{fig:nullcline_manifold1}(b) where $W^{s}(E_0)$ is below $W^{u}(E_1)$. Thus the stability of the interior equilibrium is altered and not all positive solutions tend toward $E_0$.

\begin{figure}[H]
\begin{center}
\subfigure[]{
   \includegraphics[scale=.145]{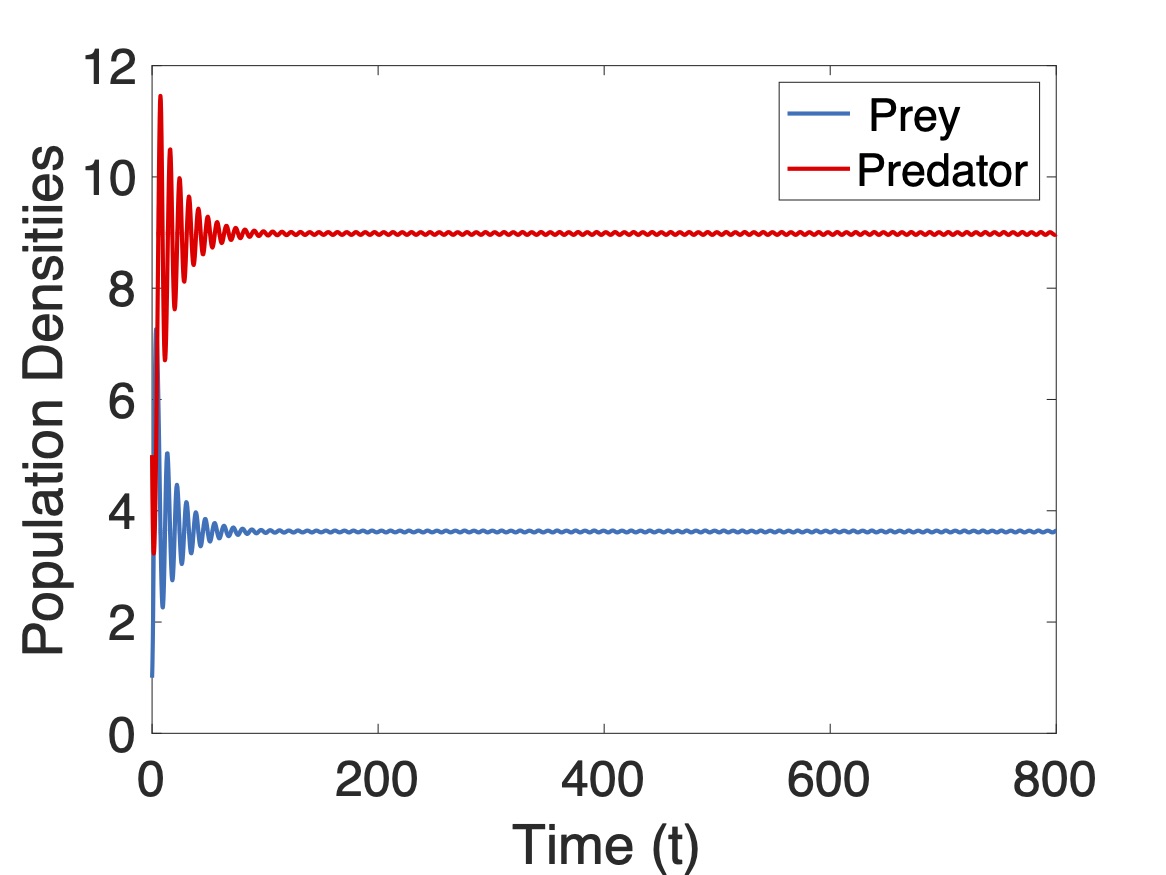}}
\subfigure[]{
    \includegraphics[scale=.145]{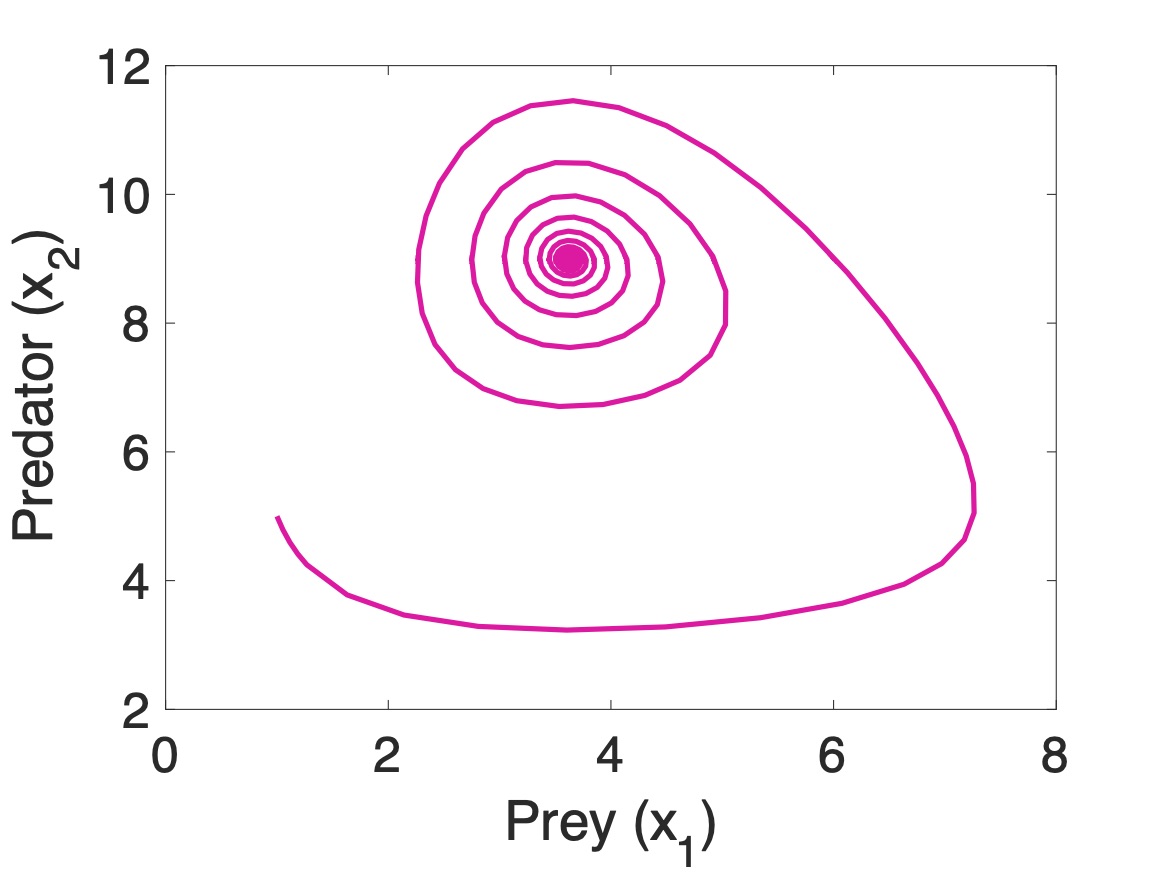}}
    \subfigure[]{
    \includegraphics[scale=.13]{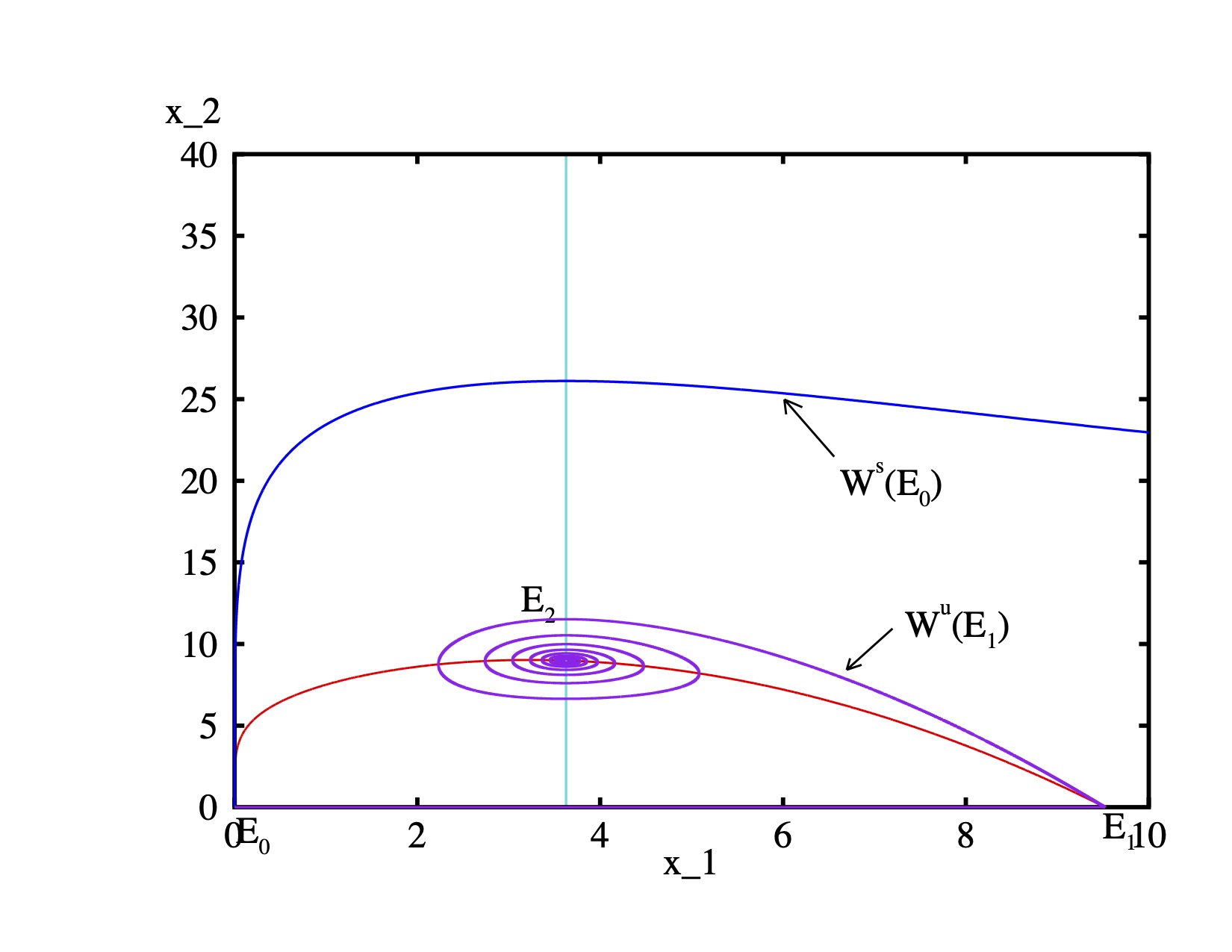}}
\end{center}
 \caption{Stable dynamics (a) time series (b) phase diagram of the interior equilibrium point ($4.83648,9.52147$) (c) $E_2$ is an attractor and $W^s(E_0)$ is above $W^u(E_1)$. Here $r=0.4$ and all other parameter sets are given in Table \ref{table:Paramset2}}
      \label{fig:time_phase1}
\end{figure}

\begin{figure}[H]
\begin{center}
\subfigure[]{
   \includegraphics[scale=.10]{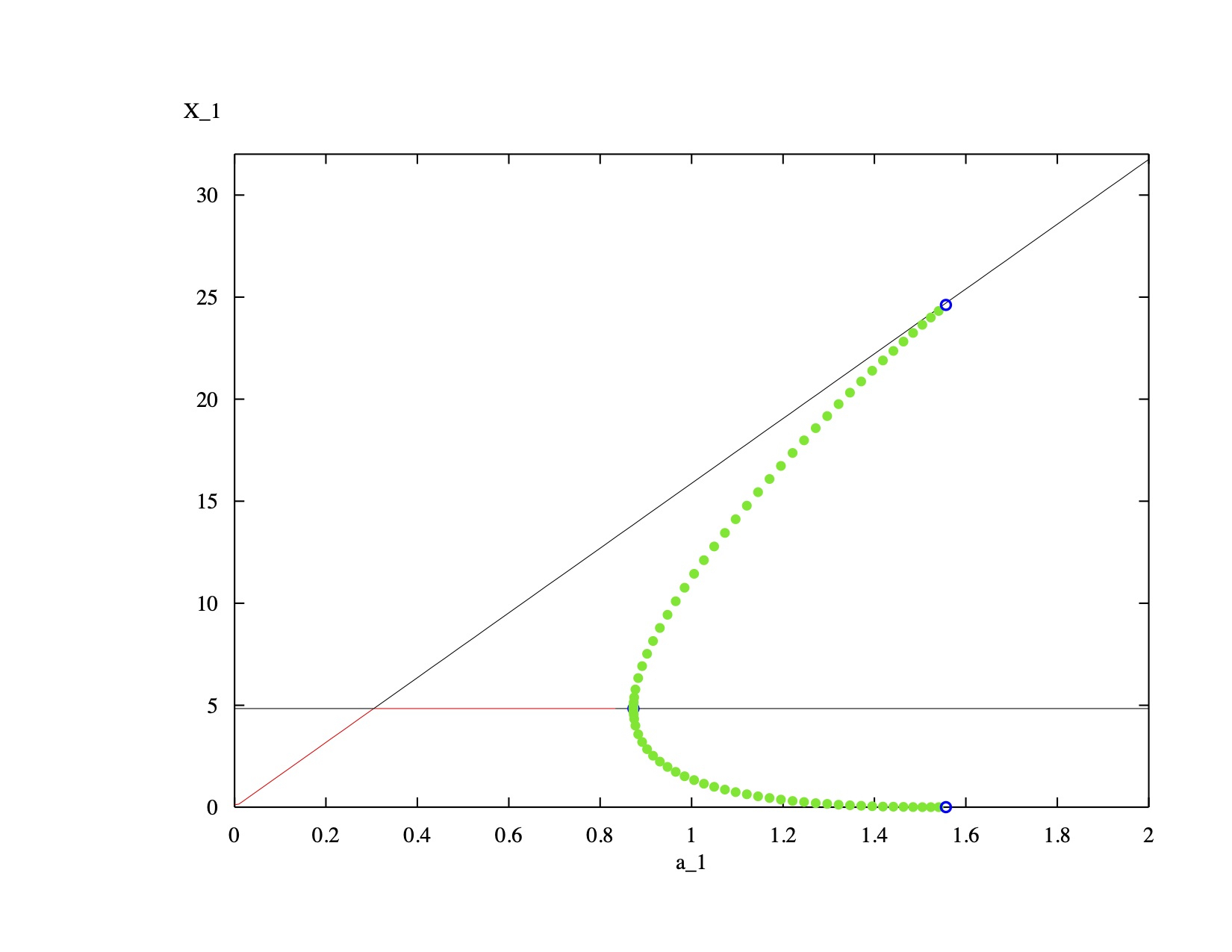}}
\subfigure[]{
    \includegraphics[scale=.10]{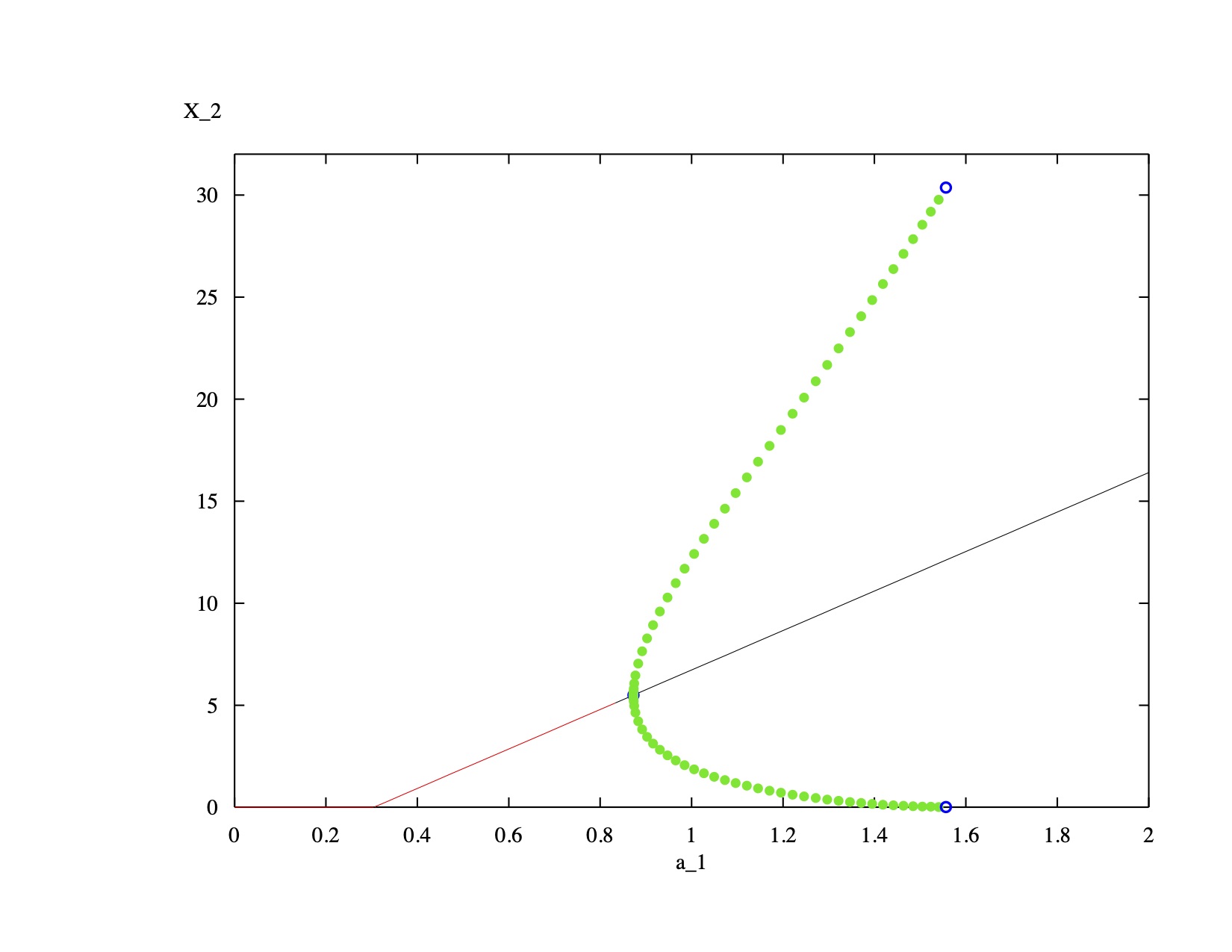}}
\end{center}
 \caption{Bifurcation diagrams of the model \eqref{refuge}, as $a_1$ crosses its critical value $a_1^*$. The stable and unstable interior equilibriums are given by the lines in red and black, respectively. The solid circles (green) represent stable limit cycles and the open circles (blue) represent unstable
limit cycles. (a) prey ($x_1$) (b) predator ($x_2$).  Parameter set are given in Table \eqref{table:Paramset2}.}
      \label{fig:bifurcation2}
\end{figure}
\begin{figure}[H]
\begin{center}
\subfigure[]{
   \includegraphics[scale=.10]{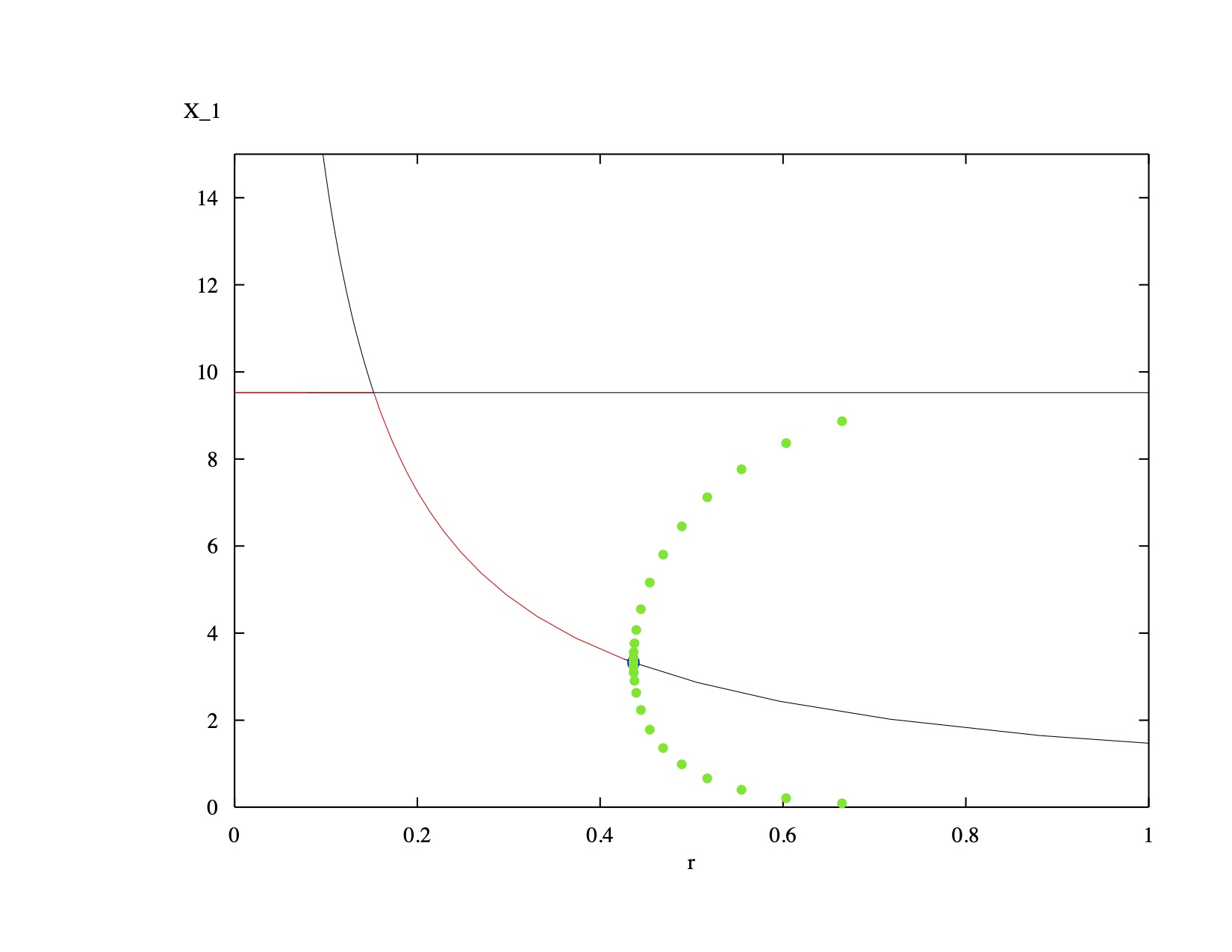}}
\subfigure[]{
    \includegraphics[scale=.10]{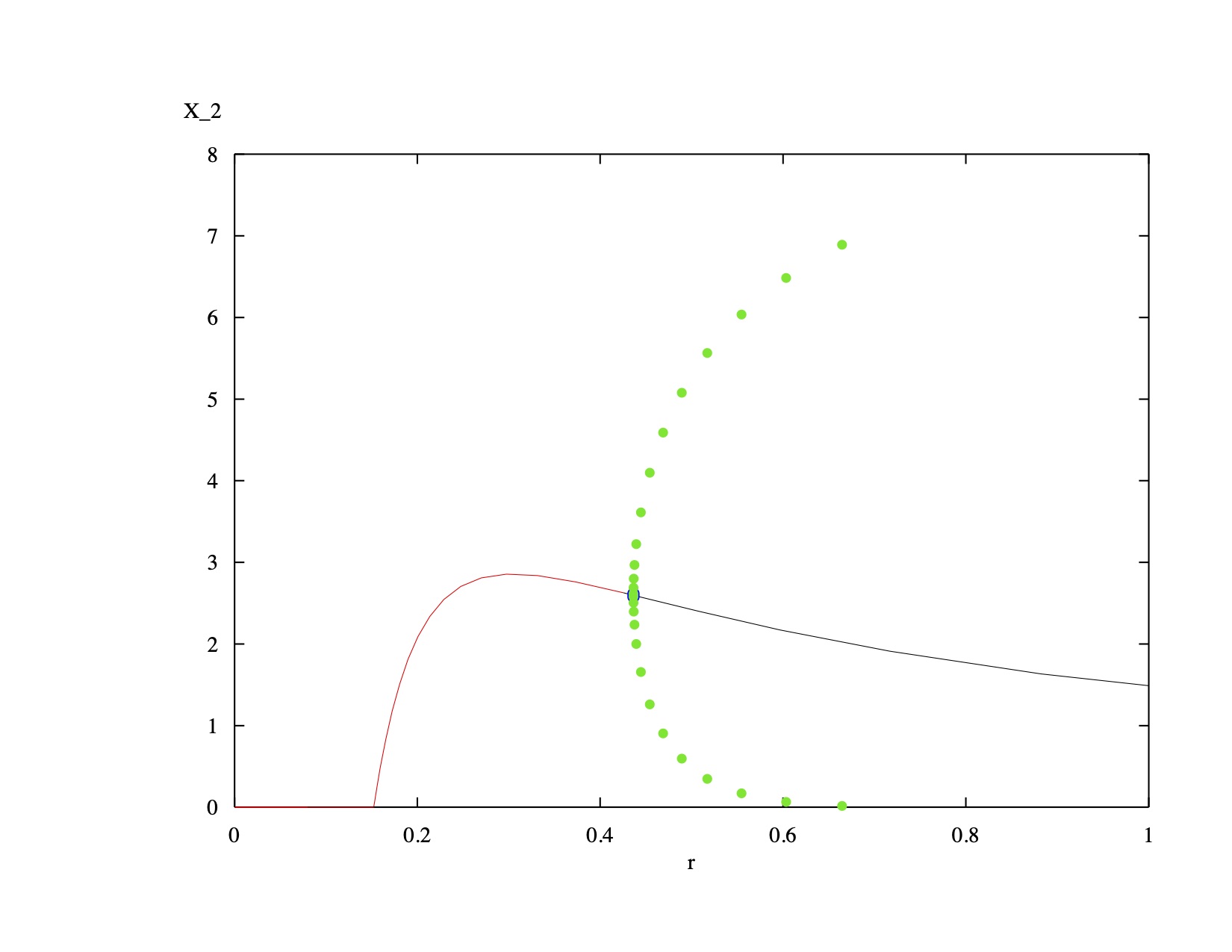}}
\end{center}
 \caption{ Bifurcation diagrams of the model \eqref{refuge}, as $r$ crosses its critical value $r_1^*$. The stable and unstable interior equilibriums are given by the lines in red and black, respectively. The solid circles (green) represent stable limit cycles. (a) prey ($x_1$) (b) predator ($x_2$).  Parameter set are given in Table \eqref{table:Paramset2}.}
      \label{fig:TB_Hopf_bifurcation_refuge}
\end{figure}

Additionally, for $r=0.3$ and all the other parameter sets provided in Table \ref{table:Paramset2}, we use  AUTO as implemented in the continuation software XPPAUT to analyze the bifurcation diagrams of the model \eqref{refuge} in  Fig. \ref{fig:bifurcation2}. The model undergoes Hopf-bifurcation  around $E_2(4.83648,5.49507)$  as the parameter  $a_1$ crosses its critical value $a_1^*=0.87278$. The branch of periodic orbits emitting from $a_1^*$ are stable and the first Lyapunov coeffiicient \cite{Perko13}, $\sigma=-2.88256e^{-3}<0$, hence the Hopf-bifurcation is supercritical.

Furthermore, the model undergoes Hopf-bifurcation around $E_2(3.32486,2.59694)$ as the parameter $r$ crosses its critical value $r_1^{**}=0.43639$, see Fig. \ref{fig:TB_Hopf_bifurcation_refuge}. The branch of periodic orbits bifurcation from $r_1^{**}$ are stable and the first Lyapunov coefficient is $\sigma=-4.99384e^{-3}$, hence supercritical. Also, the model \eqref{refuge} undergoes transcritical bifurcation around $E_1(9.52381,0)$ when the parameter $r$ crosses its threshold $r_1^{*}=0.15239$, see Fig. \ref{fig:TB_Hopf_bifurcation_refuge}(a) .
%
\begin{figure}[H]
\begin{center}
    \includegraphics[scale=.20]{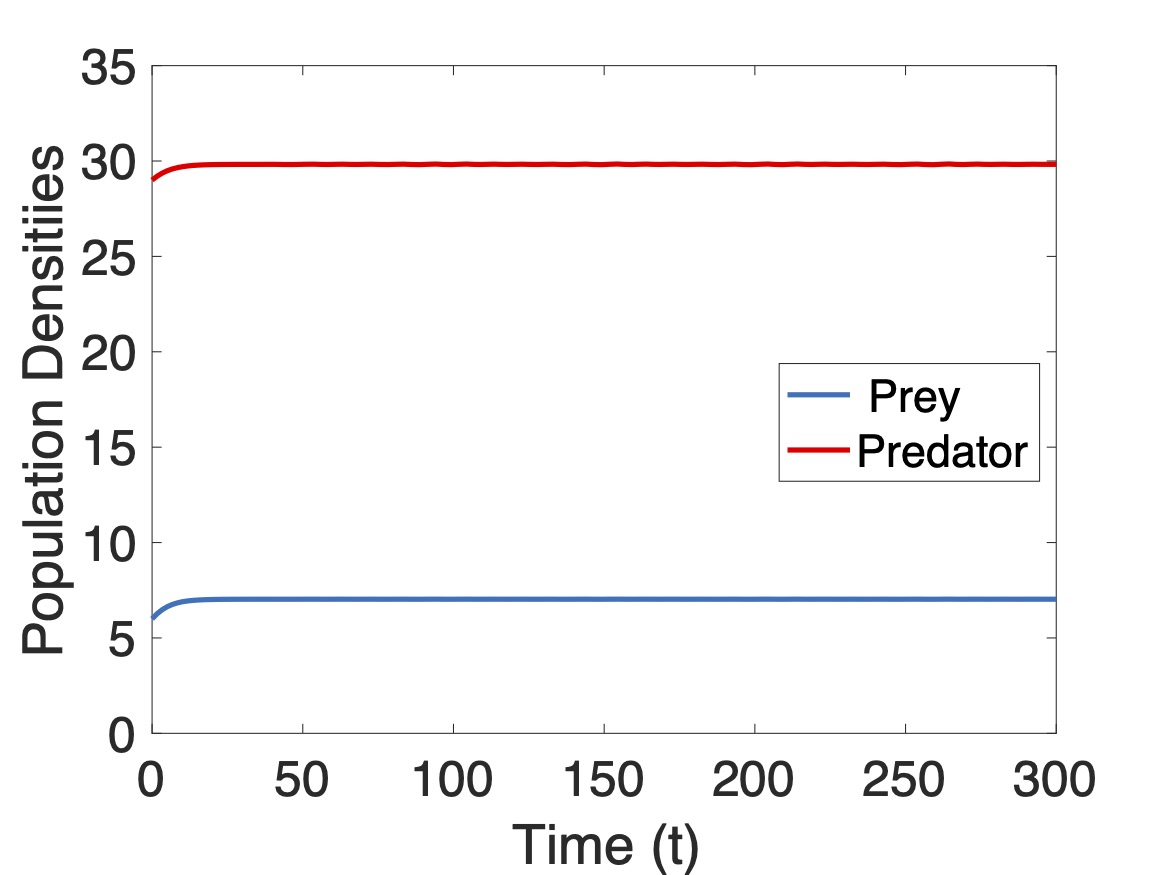}
\end{center}
 \caption{Time series depicting the stability behavior of the interior equilibrium point $(7.03041,29.8249)$ for $r=0.3$ and all other  parameter sets are given in Table \ref{table:Paramset3}.  }
      \label{fig:time_series_refuge_SN}
\end{figure}
\begin{figure}[!htb]
\begin{center}
    \subfigure[]{
    \includegraphics[scale=.09]{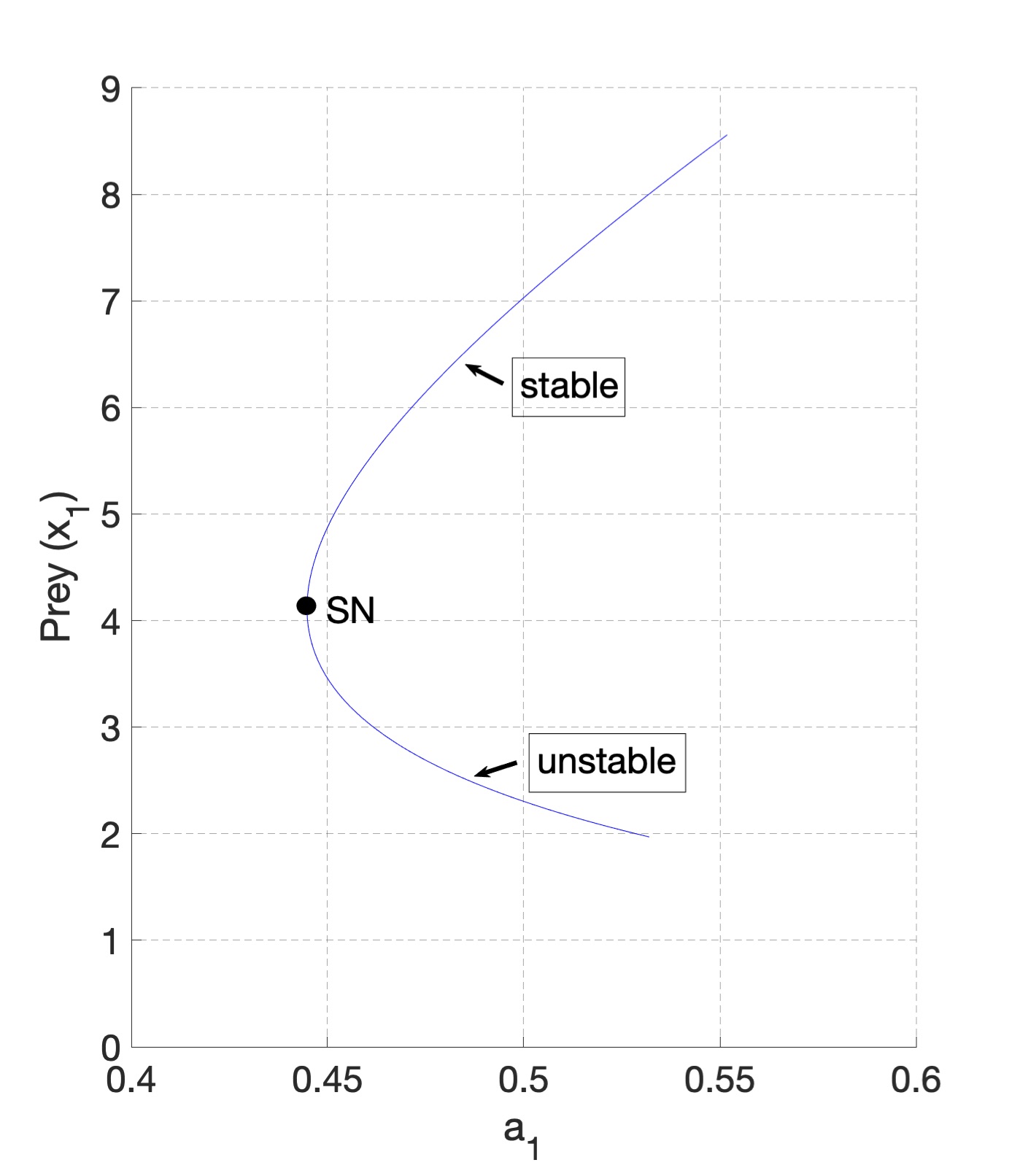}}
    \subfigure[]{
    \includegraphics[scale=.09]{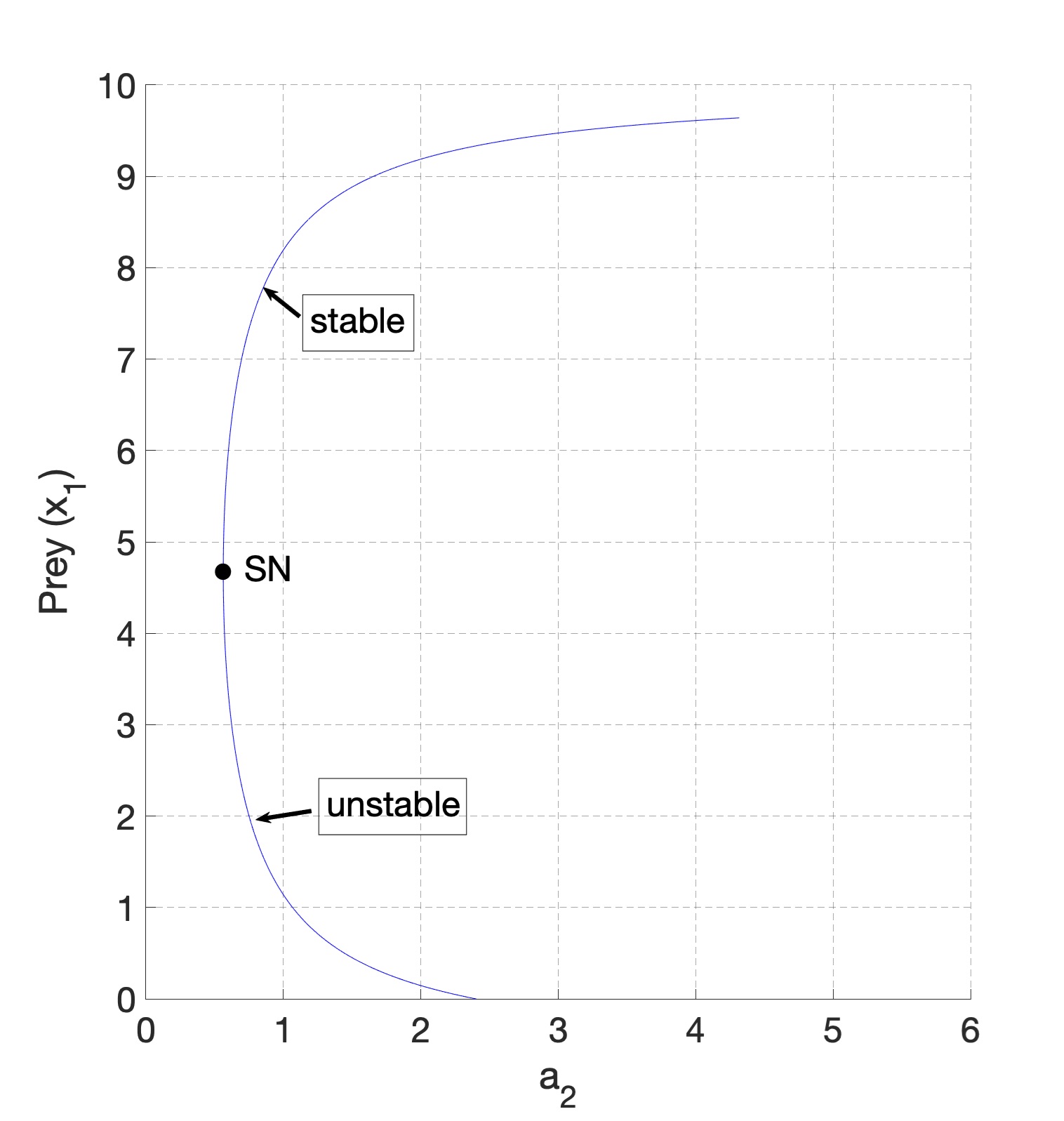}}
    \subfigure[]{
    \includegraphics[scale=.09]{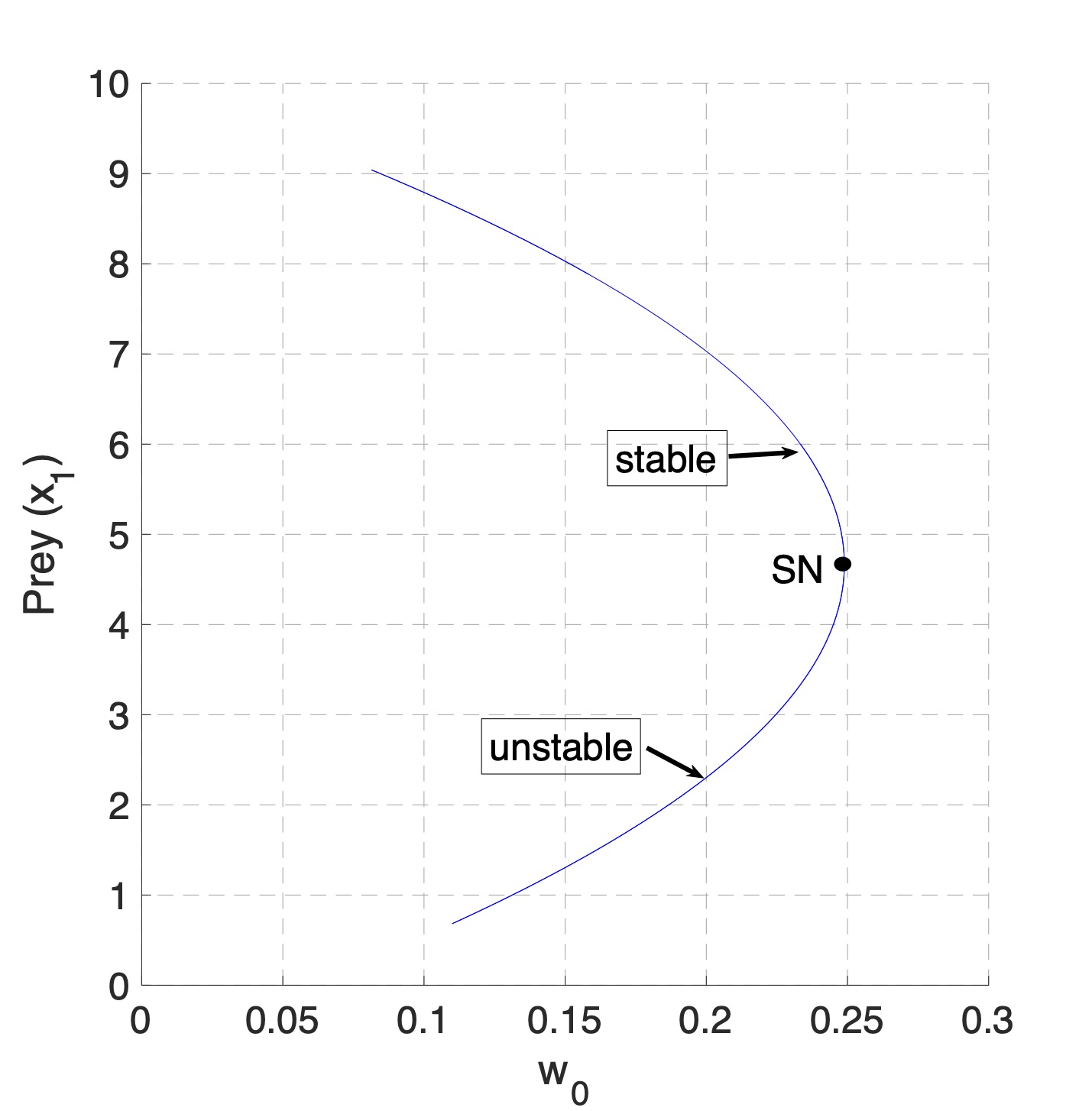}}
   \subfigure[]{
    \includegraphics[scale=.09]{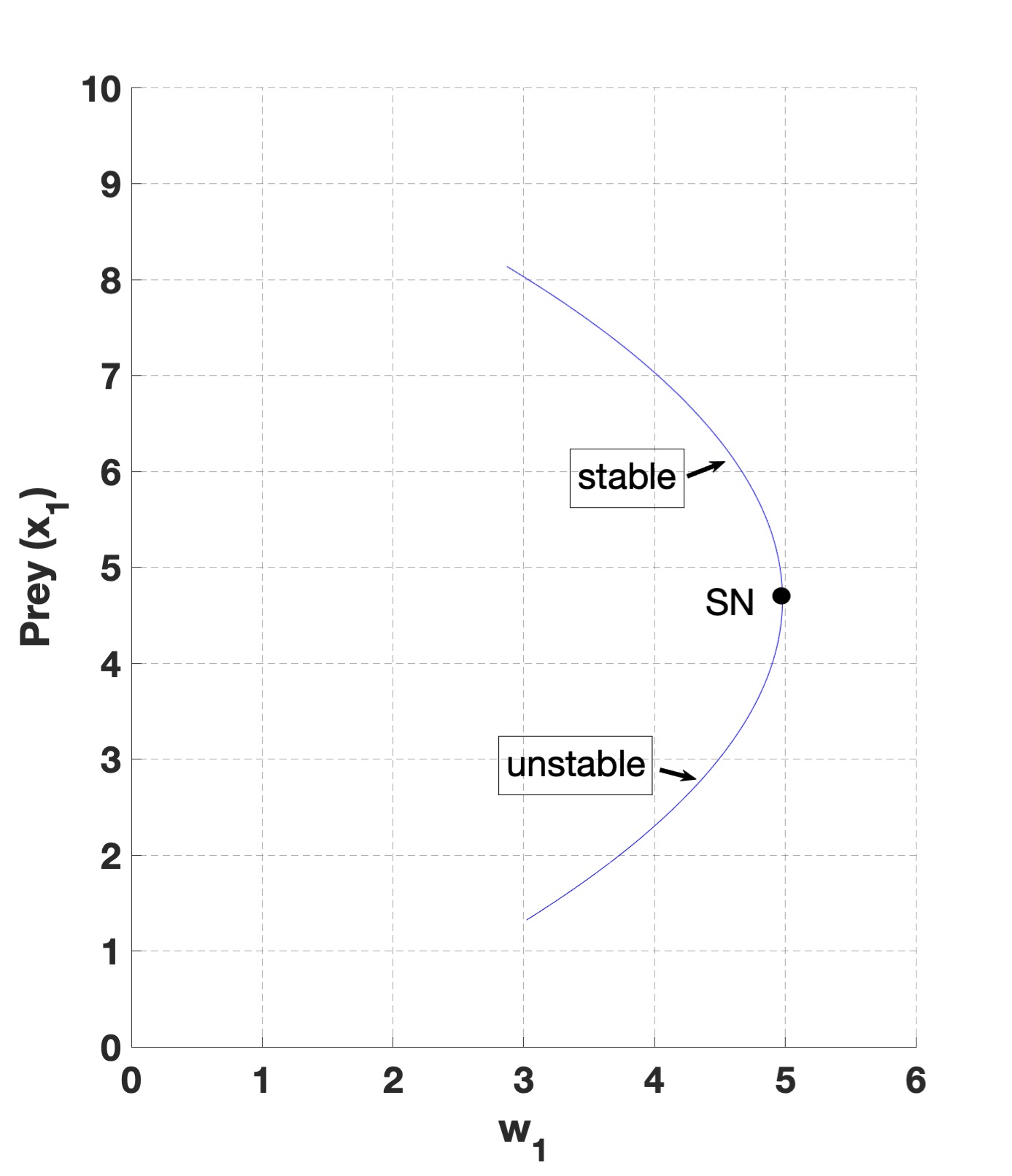}}
  \subfigure[]{
    \includegraphics[scale=.09]{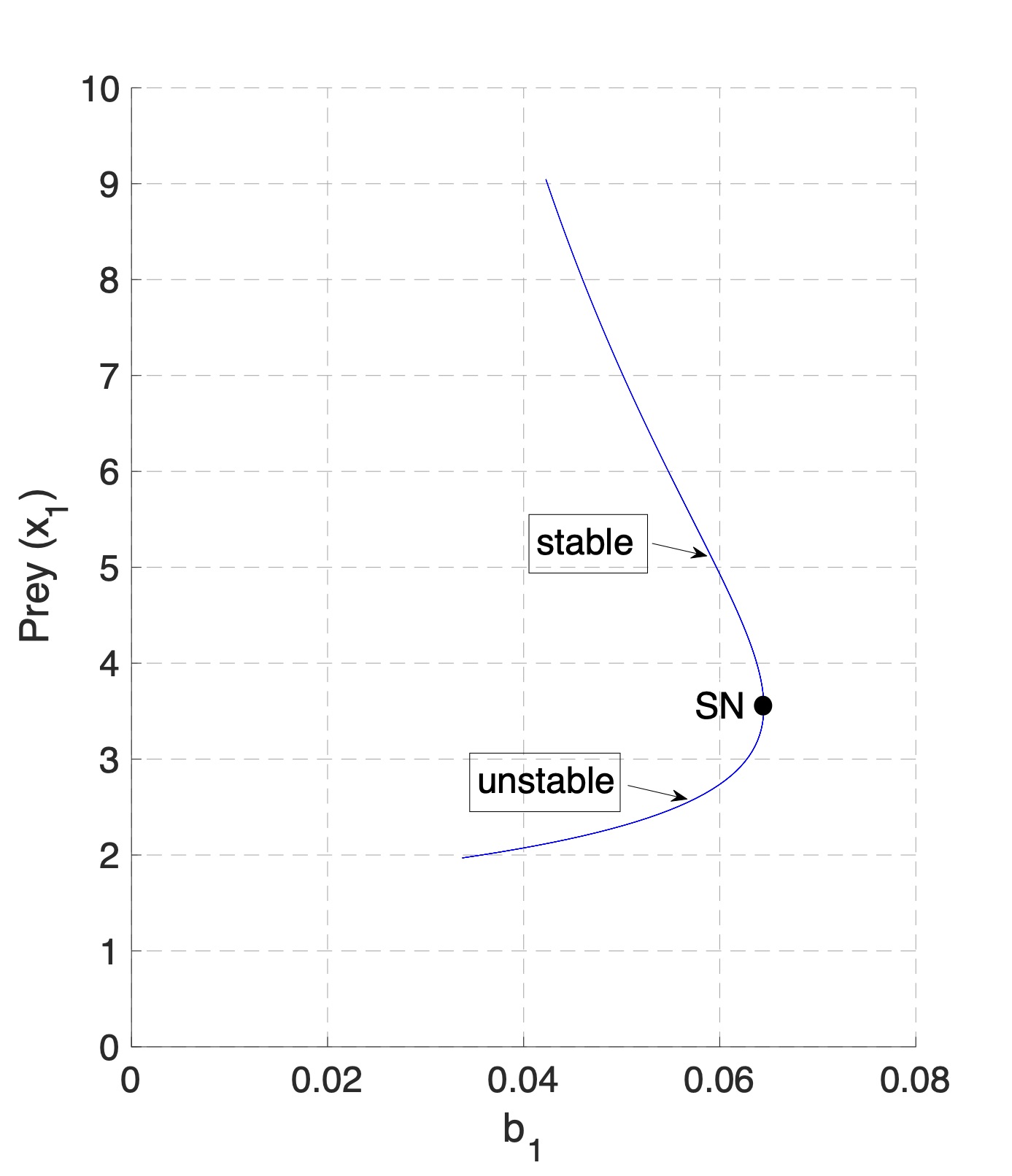}}
\end{center}
 \caption{ Bifurcation diagrams of the model \eqref{refuge} illustrating (a) SN at $a_1=a_1^*=0.44476$, (b) SN at $a_2=a_2^*=0.5625$, (c) SN at $w_0=w_0^*=0.24889$, (d) SN at $w_1=w_1^*=4.97778$,  (e) SN at $b_1=b_1^*=0.064498$. Here $r=0.3$ and other parameter sets are given in Table \ref{table:Paramset3}. (SN: Saddle-node bifurcation.)  }
      \label{fig:SN_refuge}
\end{figure}
Next, for $r=0.3$ and the other parameter sets given in Table \ref{table:Paramset3} of the model \eqref{refuge}, we obtain the following two interior equilibrium points $E_2^1(2.30292,25.3225)$ and $E_2^2(7.03041,29.8249)$ and the predator-free equilibrium point is $E_1(10,0)$. The eigenvalues associated with $E_2^1(2.30292,25.3225)$ are $-0.32246$ and $0.19897$, hence $E_2^1$ is a saddle. The eigenvalues associated with $E_2^2(7.03041,29.8249)$ are $-0.33157$ and $-0.22790$, hence $E_2^2$ is locally asymptotically stable, see Fig. \ref{fig:time_series_refuge_SN}. We note here that, $E_1(10,0)$ cannot be analyzed using the linear stability method since $m_1=m_2=0.5<1$. We observed that the model \eqref{refuge} undergoes a saddle-node bifurcation around $E_2(x_1^*,x_2^*)$ when the following bifurcation parameters $a_1,a_2,w_0,w_1$ and $b_1$ crosses their corresponding critical values $a_1^*=0.44476,a_2^*=0.5625,w_0^*=0.24889,w_1^*=4.97778,$ and $b_1^*=0.064498$ respectively. The saddle-node bifurcation diagrams are presented in  Fig. \ref{fig:SN_refuge}.
%
\section{Discussions and Conclusions}\label{section:Dissussion_conclusion}
In this work, we consider a predator-prey model, that allows us to model both the feeding intensity of the predator, as well as the effect of mutual/predator interference. Through numerical simulations, it has been noticed that based on the non-uniqueness of the solutions of model \eqref{EquationMain}, when $m_1<1$ and $m_2=1$, the interior equilibrium $E_2$ is not globally asymptotically stable when it is an attractor (see Fig. \ref{fig:nullcline_manifold2}). We observe that the per capita rate of self-reproduction $a_1$ plays an important role because the interior equilibrium point $E_2$ changes stability at the bifurcation point $a_1^*$ (see Fig. \ref{fig:bifurcation1}). The limit cycle through the bifurcation point is stable  hence a supercritical Hopf-bifurcation.

 Furthermore, the effect of prey refuge is also considered in model \eqref{refuge} - thus one can see the interplay of all of these factors in this model. The model possesses a rich array of dynamical behavior. We have established analytically the occurrence of various local bifurcations including saddle-node, transcritical and Hopf bifurcations. The occurrence of these local bifurcations are well supplemented with one parameter bifurcation diagrams (see Figs. \ref{fig:bifurcation2}, \ref{fig:TB_Hopf_bifurcation_refuge}, and \ref{fig:SN_refuge}).  Prey extinction in finite time is also possible - for large enough initial predator density, and small enough initial prey density. Moreover,  we observed that when $W^s(E_0)$ is above $W^u(E_1)$, all solutions with initial conditions above $W^s(E_0)$ goes to prey extinction in finite time (see Fig. \ref{fig:nullcline_manifold1}(a)). This is in line with the result in \cite{BS19}. Thus, from a practical point of view increasing $m_{1}$ or decreasing the feeding intensity of the predator, will maintain ecosystem balance, as this decreases the predator nullcline, decreasing predator numbers and increasing prey numbers. 
 
 Stability in the system can also be maintained via provision of the prey with refuge. This is rigorosly established via theorem \ref{thm:ref1}. The requisite condition for a critical refuge, for persistence, derived via the theorem sheds light on various ecological scenarios. The ecological validity of the prey extinction state $(0, x^{*}_{2})$ is questionable. In the experiments of Gause \cite{G34}, once the prey has gone extinct the predator population also crashes, as there is no alternative/additional food in the experimental system. In a real scenario however, such a state might be indicative of a predator having switched to another food source after its primary source has depleted or surviving on additional food, such as in a bio-control situation \cite{P19}.


\section*{Conflict of Interest}
The authors declare there is no conflict of interest in this paper.

\bibliographystyle{mdpi}



\begin{thebibliography}{999}
\bibitem{Hassell71}
\newblock Hassell, M. (1971) Mutual interference between searching insect parasites. J. Anim. Ecol. 40, 473-486.

\bibitem{Hassell75}
\newblock Hassell, M. (1975) Density dependence in single species population. . J. Anim. Ecol. 44, 283-295.

\bibitem{Upadhyay15}	
\newblock Upadhyay, R.K., Agrawal, R. (2015) Modeling the effect of mutual interference in a delay-induced predator-prey system, J. Appl. Math.  Comput. 49, 13-39.

\bibitem{Upadhyay19}	
\newblock Upadhyay, R.K., Parshad, R., Antwi-Fordjour, K., Quansah, E., Kumari, S. (2019) Global dynamics of stochastic predator-prey with mutual interference and prey defense. J. Appl. Math. Comput. 60, 169-190.

\bibitem{Ermentrout02}	
\newblock Ermentrout, B. (2002) Simulating, analyzing, and animating dynamical systems: a guide to XPPAUT for researchers and students. vol 14, SIAM.

\bibitem{BS19}	
\newblock Beroual, N., Sari, T. (2019) A predator-prey system with Holling-type functional response. hal-02002894.

\bibitem{Matcont}
\newblock Dhooge, A., Govaerts, W., Kuznetsov, Yu. A., Meijer, H.G.E., Sautois, B. (2009) New features of the software MatCont for bifurcation analysis of dynamical systems, MCMDS, Vol. 14, No. 2, pp 147-175.

\bibitem{Murray93}	
\newblock Murray, J.D. (1993) Mathematical biology, Springer, New York.

\bibitem{Perko13}	
\newblock Perko, L. (2013) Differential equations and dynamical systems. Vol. 7, Springer Science \& Business Media.


\bibitem{K96} Křivan, V. (1996). Optimal foraging and predator–prey dynamics. theoretical population biology, 49(3), 265-290.


\bibitem{K10} V. Krivan,
\newblock \emph{Evolutionary stability of optimal foraging: partial preferences in the diet and patch models},
\newblock J Theor Biol, 267:486-494, 2010.


\bibitem{L12}
\newblock McKenzie, H. W., Merrill, E. H., Spiteri, R. J., $\&$ Lewis, M. A.
\newblock  How linear features alter predator movement and the functional response. Interface focus, 2(2), 205-216, 2012.

\bibitem{M04}
\newblock Mols, C. M., van Oers, K., Witjes, L. M., Lessells, C. M., Drent, P. J.,  $\&$ Visser, M. E. . \newblock Central assumptions of predator–prey models fail in a semi–natural experimental system. \newblock Proceedings of the Royal Society of London B: Biological Sciences, 271(Suppl 3), S85-S87, 2004.

\bibitem{R05}
\newblock Ruxton, G. D.
 \newblock Increasing search rate over time may cause a slower than expected increase in prey encounter rate with increasing prey density.
 \newblock Biology letters, 1(2), 133-135, 2005.

\bibitem{I08} Christos C. Ioannou, Graeme D. Ruxton, Jens Krause, Search rate, attack probability, and the relationship between prey density and prey encounter rate, Behavioral Ecology, Volume 19, Issue 4, July-August 2008, Pages 842–846, https://doi.org/10.1093/beheco/arn038
\newblock Bulletin of Mathematical Biology, 73(10):2249-2276, 2011.



\bibitem{K07}
Y. Kuang,
 \newblock \emph{Some mechanistically derived population models},
\newblock Math. Biosci. Eng, 4(4), 1-11, 2007.


\bibitem{K05}
Kar, T. K.
 \newblock \emph{Stability analysis of a prey–predator model incorporating a prey refuge},
\newblock Communications in Nonlinear Science and Numerical Simulation, 10(6), 681-691, 2005.

\bibitem{DV11}
DeLong, J. P., $\&$ Vasseur, D. A. 
 \newblock \emph{Mutual interference is common and mostly intermediate in magnitude},
 \newblock BMC ecology, 11(1), 2011.
 
 \bibitem{K11}
 Křivan, V.  
  \newblock \emph{On the Gause predator–prey model with a refuge: a fresh look at the history}. 
  Journal of theoretical biology, 274(1), 67-73, 2011.

 \bibitem{P16}
Parshad, R. D., Quansah, E., Black, K., $\&$ Beauregard, M. 
 \newblock \emph{ Biological control via “ecological” damping: an approach that attenuates non-target effects}, \newblock Mathematical biosciences, 273, 23-44.

\bibitem{B75}
J. R. Beddington,
 \newblock \emph{Mutual interference between parasites or predators and its effect on searching efficiency}
  \newblock J. Animal Ecol., 44, 331-340, 1975.

\bibitem{D75}
DeAngelis D. L., R. A. Goldstein and R. V. ONeill
 \newblock \emph{A model for trophic interaction},
Ecology, 56, 881-892, 1975.

\bibitem{G34}
Gause, G. F.
 \newblock \emph{The struggle for existence},
 \newblock Williams $\&$ Wilkins, Baltimore, Maryland, USA, 1934.

\bibitem{H59} Holling, C.S.
 \newblock \emph{The components of predation as revealed by a study of small mammal
predation of the European pine sawfly},
\newblock  Canad. Entomol. 91, 293-320, 1959.

\bibitem{E85} Erbe, L. H., $\&$ Freedman, H. I. 
 \newblock \emph{Modeling persistence and mutual interference among subpopulations of ecological communities}
 \newblock . Bulletin of mathematical biology, 47(2), 295-304, 1985.


\bibitem{L25} Lotka A. J.
 \newblock \emph{Elements of physical biology}.
 \newblock Williams and Wilkins, Baltimore. Reprinted
as Elements of mathematical biology, Dover, New York, 1925.

\bibitem{B12}
Braza, P. A.
 \newblock \emph{Predator–prey dynamics with square root functional responses}
 \emph Nonlinear Analysis: Real World Applications, 13(4), 1837-1843, 2012.

\bibitem{W13}
Wang, K. and Zhu, Y.
\newblock \emph{Periodic solutions, permanence and global attractivity of a delayed impulsive prey–predator system with mutual interference},
\emph Nonlinear Analysis: Real World Applications, 14(2), 1044-1054, 2013.

\bibitem{P19}
Parshad, R. D., Wickramsooriya, S., $\&$ Bailey, S. 
\newblock \emph{A remark on “Biological control through provision of additional food to predators: A theoretical study”[Theor. Popul. Biol. 72 (2007) 111–120].}
 \emph  Theoretical Population Biology, 2019.





\bibitem{F79}
Freedman, H. I.
\emph{Stability analysis of a predator-prey system with mutual interference and density-dependent death rates}, \emph Bulletin of Mathematical Biology, 41(1), 67-78, 1979.

\bibitem{U09}
Upadhyay, R. K., and Rao, V. S. H.
\emph{Short-term recurrent chaos and role of Toxin Producing Phytoplankton (TPP) on chaotic dynamics in aquatic systems},
\emph Chaos, Solitons and Fractals, 39(4), 1550-1564, 2009.

\bibitem{S97}
Sugie, J., Kohno, R., and Miyazaki, R.
\emph{On a predator-prey system of Holling type},
\emph Proceedings of the American Mathematical Society, 125(7), 2041-2050, 1997.

\bibitem{S99}
Sugie, J. and Katayama, M., \emph{Global asymptotic stability of a predator–prey system of Holling type}, Nonlinear Anal-Theor., Vol. 38, Iss. 1, 105-121, 1999.

\bibitem{F06}
Finke, D. L., and Denno, R. F. (2006).
\emph{Spatial refuge from intraguild predation: implications for prey suppression and trophic cascades}.
\emph Oecologia, 149(2), 265-275.




\end{thebibliography}

\end{document}